\documentclass{article}

\usepackage{latexsym}
\usepackage{amsmath,enumerate}
\usepackage[psamsfonts]{amssymb}
\usepackage{amsthm}
\usepackage[mathscr]{eucal}
\usepackage{tikz}
\usetikzlibrary{arrows}

\newtheorem{Theorem}{Theorem}
\newtheorem{Proposition}{Proposition}
\newtheorem{Lemma}{Lemma}
\newtheorem{Remark}{Remark}
\newtheorem{Corollary}{Corollary}
\newtheorem{Definition}{Definition}

\newtheorem{Assumptions}{Assumptions}

\numberwithin{equation}{section}
\def\1{\raisebox{2pt}{\rm{$\chi$}}}
\def\a{{\bf a}}
\def\q{{\bf q}}

\def\z{{\bf z}}
\def\v{{\bf v}}
\def\R{\mathbb{R}}

\def\div{\mbox{div}\,}
\def\sign{\mbox{sign}\, }

\renewcommand{\L}{\mathcal{L}}
\renewcommand{\H}{{\mathcal H}}

\newcommand{\res}               {\!\!\mathop{\hbox{
                                \vrule height 7pt width .5pt depth 0pt
                                \vrule height .5pt width 6pt depth 0pt}}
                                \nolimits}

\title{Analysis of a class of degenerate parabolic equations with saturation mechanisms}

\author{Juan Calvo\thanks{
       Centre de Recerca Matem\`atica, Edifici C, Campus de Bellatera, 08193 Bellaterra (Barcelona), Spain.
 ({\tt jcalvo@crm.cat}).}
 }
\begin{document}

 \maketitle

\begin{abstract}
We analyze a family of degenerate parabolic equations with linear growth Lagrangian having the form $u_t=\div (\varphi(u)\psi(\nabla u/u))$. Here $|\psi|\le 1$ and saturates at infinity. We present a simple and natural set of assumptions on the functions $\psi,\varphi$, under which: 1) these equations fall in the framework provided by \cite{ACMEllipticFLDE, ACMMRelat} and hence they are well posed, 2) we can ensure finite propagation speed for these models, 3) a Rankine--Hugoniot analysis on traveling fronts is also performed. On the particular case of $\varphi(u)=u$ we get more detailed information on the spreading rate of compactly supported solutions and some interesting connections with optimal mass transportation theory.
\end{abstract}

\section{Introduction}
The purpose of this article is to analyze a certain class of degenerate parabolic equations having the following general form:
\begin{equation}
\label{mgcase}
u_t = \div \left( \varphi(u) \psi(\nabla u/u)\right),
\end{equation}
under a certain set of assumptions on $\psi,\varphi$. Such equations arise in a number of interesting situations in several branches of mathematical physics, as we detail below. This includes in particular the ``relativistic heat equation'' \cite{Rosenau92}
\begin{equation}
\label{rhe}
\frac{\partial u}{\partial t} = \nu \, \div \left( \frac{u \nabla u}{\sqrt{u^2 +\frac{\nu^2}{c^2} |\nabla u|^2}}\right)
\end{equation}
and some of its porous media variants \cite{ACMSV,leysalto}
\begin{equation}
\label{mrhe}
\frac{\partial u}{\partial t} = \nu \, \div \left( \frac{|u|^m \nabla u}{\sqrt{u^2 +\frac{\nu^2}{c^2} |\nabla u|^2}}\right).
\end{equation}

We will be chiefly interested in the following subclass of \eqref{mgcase}: 
\begin{equation}
\label{template0}
u_t = \div \left( u\, \psi(\nabla u/u)\right).
\end{equation}
Note that the prime example of \eqref{template0} is the standard heat equation (also known as Fokker--Planck or diffusion equation, depending on the context), corresponding to the choice $\psi(r)=r$. As it is well known, it lacks of propagating fronts and dissolves immediately any discontinuity initially imposed. On the contrary, equations of the form \eqref{template0} enjoy the property of convecting fronts at a constant (model dependent) speed if $\psi$ provides a suitable saturation mechanism. This is the kind of behavior we will be interested in: For the sake of various applications in heat or mass transfer, plasma diffusion, and hydrodynamics (to name a few) it is reasonable to look for suitable modifications of the standard heat equation for which heat transfer proceeds by means of convected fronts at large gradient regimes, instead of sheer diffusion. As pointed out in the seminal paper \cite{Rosenau92}, the idea is to have a model that resembles the heat equation at moderate gradient size, while behaving like a hyperbolic equation at large gradient regimes. To track these large gradient regimes we may compute the relative size $|\nabla u|/u$. In most applications the spatial variable is measured in units of length. Then the ratio $|\nabla u|/u$ is measured in units of 1/length.

Pushing this idea a bit further leads us to equations of the form \eqref{template0} in a natural way and elucidates how should $\psi$ look like in order to get the desired behavior. To see how, let us focus in the case of heat flow and assume that the evolution of heat is described by means of an equation of the following form:
\begin{equation}
\label{qform}
 u_t - \div \q =0.
\end{equation}
What is suggested in \cite{Rosenau92} is to re-write \eqref{qform} as a transport equation, namely
\begin{equation}
\label{Vform}
 u_t - \div (uV) =0.
\end{equation}
The velocity $V$ may be a function depending on $x,t,u$ and its derivatives (even in a non-local way). The important idea is the following: If the equation is to convect fronts, shocks, etc, then $|V|$ must saturate to a constant value in the regime in which $|\nabla u|/u$ diverges. More specifically, $|V|$ must saturate to the speed of sound, which is the highest admissible free velocity in a medium. Recall that $V =\nu \nabla u/u$ for the case of the heat equation with diffusion coefficient $\nu$,
\begin{equation}
\label{lheat}
u_t = \nu \Delta u.
\end{equation}
We want $V$ to resemble $\nu \nabla u/u$ on the regime of moderate gradient size, while $|V|$ must converge to the speed of sound when $|\nabla u|/u$ diverges. The easiest way to achieve this is to impose $V$ to be a function of $\nabla u/u$ alone, on which we fix the limit behavior at will. A particular instance of this strategy is obtained by setting
$$
V= \frac{\nu \nabla u}{\sqrt{u^2 +\frac{\nu^2}{c^2} |\nabla u|^2}},\quad \nu,\, c>0
$$
which yields \cite{Rosenau92} the so-called ``relativistic heat equation'' \eqref{rhe} (after \cite{Brenier1}). We see that $V$ above resembles $\nu \nabla u/u$ when $|\nabla u|\ll u$, while $|V|$ converges to $c$ for $|\nabla u|/u \nearrow \infty$. Following this rationale set in \cite{Rosenau92}, there are lots of other choices for $V$ that would have nearly the same effect (at least at a formal level). The general strategy that we briefly outlined above was not pursued by Rosenau and coworkers in the subsequent series of papers \cite{Rosenau2003,Rosenau2005,Rosenau2006} (in which they discuss in particular several variants of \eqref{rhe}, some of them of the form \eqref{mgcase}). In fact, the following idea permeates these works: As long as the flux function is monotone in gradients and saturates above a certain rate, the particulars of the flux function are no that important.

Our study in the present document can be regarded as a rigorous statement of that intuitive idea. Namely, we provide a suitable general framework in which this issue can be addressed successfully. Indeed, we show that the class \eqref{template0} constitutes such an adequate framework: A number of distinctive qualitative properties hold for such class of equations when $\psi$ satisfies some suitable requirements (to be detailed in Section \ref{setas} below), irrespective of the precise function $\psi$ that is used. {First, we show that under such requirements equations of the form \eqref{template0} fall under the scope of the theory in \cite{ACMEllipticFLDE, ACMMRelat}, hence they are well posed in the class of entropy solutions. Then, combining some results and techniques in \cite{ARMAsupp, leysalto} with a number of new ideas, we are able to show that, under the aforementioned requirements:}
\begin{enumerate}
\item Equations \eqref{template0} can be formally deduced from the point of view of optimal transport theory, using cost functions with domain contained in a ball having the speed of sound as its radius.
\item The Rankine--Hugoniot relation holds for propagating discontinuities, which transverse the medium at the speed of sound.
\item The support of any solution propagates at a finite speed, bounded above by the speed of sound. Under some positivity and structure assumptions (including all the relevant examples in the literature so far), we can ensure that its spreading rate is \emph{exactly} the speed of sound.
\end{enumerate}
To the best of our knowledge, some of these properties have been stated in a rigorous way only for \eqref{rhe} \cite{ACMMRelat, ARMAsupp,Brenier1,leysalto} among all the models having the form \eqref{template0}. This makes also clear that so far there is no a priori reason to privilege the usage of \eqref{rhe} over other flux-saturated models that we may come up with. 

Although our main interest lies in \eqref{template0}, some of the previous statements have suitable generalizations to the more general case of \eqref{mgcase}. Even in some cases this is conceptually simpler, as the underlying ideas appear in a clearer way when we treat the general situation. Thus, the analysis of both types of equations will be intertwined in the sequel. 

We felt that the ideas in \cite{Rosenau92} provide a convenient way to introduce the class \eqref{template0}, but this is by no means the only place in which equations of this sort show up. In fact, some particular instances appear already in an unpublished work of J.R. Wilson concerning radiation hydrodynamics (see \cite{Mihalas}) and in the works by Levermore and Pomraning about radiative transfer \cite{Levermore79,Levermore81,Levermore84}. Let us also mention that the class of equations given by \eqref{template0} was already present in \cite{Coulombel} for the one-dimensional case, although their reasons in order to introduce it are of a quite different nature. Similar hyperbolic phenomena in a related class of degenerate parabolic equations were also observed in \cite{BdPasso}. More recently these ideas have also found some applications in mathematical biology \cite{KSlimitado,Verbeni}.

Let us detail what is the plan of the paper. In the following section we introduce the set of assumptions to be considered in order to ensure that equations \eqref{template0} satisfy the aforementioned properties. After that we state the main results of the document, Theorems \ref{Main2} and \ref{Main}, which phrase those properties in a rigorous way. The section is completed by a list of examples comprising a number of equations from the literature that fall under the present theory. Section \ref{ACM} is a summary of the well-possedness theory developed in \cite{ACMEllipticFLDE, ACMMRelat}, which introduced the fundamental notion of \emph{entropy solutions}. Although such theory is the cornerstone in which all the results of the paper are based, this section is technically involved and may be skipped in a first reading. The purpose of the remaining sections is to supply proofs for the statements in Theorems \ref{Main} and \ref{Main2}. Namely, Section \ref{wp} deals with the well-posedness of \eqref{template0}, Section \ref{Ftransport} treats the optimal transportation formulation of these problems, Section \ref{spread} analyzes the behavior of the spatial support of solutions during evolution, and Section \ref{rhcond} tackles the formulation of Rankine--Hugoniot conditions in this context. Several of the results that are proved in those sections hold under more general structure assumptions than \eqref{template0}, we will comment on this in each precise case.

Finally we state some notations that are common to the whole document. The spatial dimension is always denoted by $d$. We use $B(x,R)$ to denote an open ball of center $x\in \R^d$ and radius $R$. The Minkowsky sum of two set $A,B \subset \R^d$ will be written as $A \oplus B=\{x\in \R^d/ x=a+b \, \mbox{with}\, a\in A,\, b \in B\}$. We use $cl(\cdot)$ for the closure of a set. The indicator function of a set $A \subset \R^d$ is written as $\chi_A$. The Kronecker delta is $\delta_{ij}=1$ if $i=j$, zero otherwise. We use $|\cdot|$ to denote either the modulus of a vector or the absolute value of a number; this will be clear from the context. The scalar product of two vectors $u , v \in \R^d$ is indicated as $u \cdot v$. A superscript like $v^T$ means transposition. Given an open set $\Omega \subset \R^d$  we denote by ${\mathcal D}(\Omega)$  the space of infinitely differentiable functions with compact support in
$\Omega$. The space of continuous functions with compact support
in $\Omega$ will be denoted by $\mathcal{C}_c(\Omega)$. In a similar way, $L^p(\Omega)$ and $C^k(\Omega)$ denote Lebesgue spaces and spaces of functions of class $k$. We use $\|\cdot\|_p$ to denote the norm in $L^p(\Omega)$, the base set will be clear from the context. Given $u : \Omega \rightarrow \R$, 
$\mbox{supp}\, u$ denotes the essential support, while $u^+= \max \{u,0\}$, $u^- = - \min \{u,0\}$ are the positive and negative parts respectively. For any $T>0$, we let $Q_T:=(0,T)\times \R^d$ and we write $u=u(t,x)$ for functions defined in $Q_T$. Partial derivatives with respect to $x_i$ are abridged as $\partial_i, i=1,\ldots, d$. Sometimes we use subscripts instead, as $u_t,u_x$ and the like. Finally, $O()$ and $o()$ are the standard Landau symbols, while $\sim$ indicates asymptotic equivalence.

\section{Structure assumptions and main results}
\label{setas}
Henceforward we deal only with non-negative solutions, which are the relevant ones for the applications. The main idea is that $|\psi(\nabla u/u)| \rightarrow \psi_\infty \in \R^+$ when $|\nabla u/u| \to \infty$. In particular, $|\psi|$ cannot be a power law. Hence, we write our template in the following way\footnote{In fact we should write it as 
$$
\frac{\partial u}{\partial t} = \div \left( s\,|u|\,  \psi \left(L\frac{\nabla u}{|u|}\right)\right)
$$
 in order to deal with signed solutions. Since we are chiefly interested in non-negative solutions, we will make a slight abuse of notation and refer always to \eqref{template} as the way of writing down the equation and specific examples. Similar conventions hold for \eqref{template2} below.}
\begin{equation}
\label{template}
\frac{\partial u}{\partial t} = \div \left( s\,u\,  \psi \left(L\frac{\nabla u}{u}\right)\right).
\end{equation}
Here $L>0$ is a constant having dimensions of length and $s$ a constant having dimensions of speed. Thus $s$ can be regarded as a characteristic speed. Note that $s:=c$ and $L:=\nu/c$ for the case of \eqref{rhe}. Rosenau terms $c$ as the speed of sound in \cite{Rosenau92}. It is the maximum speed of propagation that is allowed in the medium, further justified by optimal transport interpretations of the equation \cite{Brenier1,McCann}.

If we compare \eqref{template} with \eqref{Vform} we find out that
\begin{equation}
\label{estar}
V = s\, \psi \left(L\frac{\nabla u}{u}\right).
\end{equation}
Then $|V|$ would converge to $s \psi_\infty$ whenever $|\nabla u/u| \to \infty$. Thus, there is no loss of generality in assuming that $\psi_\infty =1$ (otherwise we rescale the effective speed). These comments are the main reason for the  list of assumptions on the function $\psi$ below.

Before introducing such list of assumptions, we note that a more general class of equations can be considered following the same guidelines. Namely, let
\begin{equation}
\label{template2}
\frac{\partial u}{\partial t} = \div \left( {\varphi}(u)\,  \psi \left(L\frac{\nabla u}{u}\right)\right)
\end{equation}
where ${\varphi}$ is customarily an even, non-negative convex function such that dimensions fit (we will be more precise about this below); think for instance in ${\varphi}(u) = |u|^m$ for $m>1$ (in dimensionless form) \cite{leysalto}. Note that \eqref{template2} is a generalization of \eqref{template}, which falls under \eqref{template2} for the particular choice $\varphi(z)=s z$. The ``velocity'' is now given by 
$$
V= \frac{\varphi(u)}{u} \psi \left(L\frac{\nabla u}{u}\right).
$$
Let us state now what will be required of $\psi$ in order to build up a reasonable theory. Just before proceeding, note that we may scale out the lengthscale $L$ by letting $\hat x=x/L$. Then, without any loss of generality we assume $L=1$ for the rest of the document, except at some places in which we found useful to keep the original lengthscale.
\begin{Assumptions}
Let $\psi=(\psi^{(1)},\ldots,\psi^{(d)}) : \R^d \rightarrow \R^d$ enjoy the following properties
\label{as1}
\begin{enumerate}
\item $\psi \in C^1(\R^d,\R^d)$. \label{suave}
\item $\psi(0)=0$. \label{nulo}
\item 
\label{infinito}
$|\psi(r)-r/|r||\le d(|r|)$ for any $r \in \R^d$ and for some continuous function $d:\R_0^+\rightarrow \R_0^+$ such that $\lim_{r\to \infty} d(r)=0$. Thus $\lim_{|r|\to \infty} |\psi(r)|= 1$.
\item \label{parity}
If $d=1$ we require that 
\begin{enumerate}
\item $\psi$ be odd and monotonically increasing,
\item \label{1dcaida} {$|\psi'(r)|=o(1/|r|)$ for $|r|\gg 1$},
\end{enumerate}
 while the following properties are required for dimension greater than one:
 \begin{enumerate}
 \item $\psi$ is a conservative vector field,
 \item \label{minuse} $\psi(-r)=-\psi(r)\quad \forall r \in \R^d$,
  \item \label{convex}
  The Jacobian matrix of $\psi$, $D\psi$, 
  is a non-negative definite (symmetric) matrix.
  \item \label{caida} 
 $\|D\psi\|_\infty(r) = O(1/|r|)$ for $|r|\gg 1$.
 \end{enumerate}

 \end{enumerate}
\end{Assumptions}

\noindent
Assumptions \ref{as1} enable to describe the behavior of our models in the large gradient regime in a coarse way and are independent of the precise form of the function $\varphi$.
Note that the fact that these models are nearly isotropic in the regime in which $|\nabla u/u|$ is large is implied. Apart from this, we would find physically reasonable to have $|V|\le s$ in \eqref{estar} --which is the case for all the models of interest, see below--, but this may not be necessarily implied by Assumptions \ref{as1}. What is true is that the radial component of the velocity $V$ is bounded by $s$.
\begin{Lemma}
\label{tangential-radial}
Being Assumptions \ref{as1} verified, the following assertions hold true:
\begin{enumerate}
\item Fix $r\in \R^d$. Then the map $t \mapsto \psi(t r)\cdot r,\, t>0$ is non-decreasing and hence non-negative.
\item $|\psi(r)\cdot r|\le |r|$ for every $r\in \R^d$. \label{dos}
\end{enumerate}
\end{Lemma}
\begin{proof}
First item is a consequence of Assumptions \ref{as1}.\emph{\ref{convex}} and \ref{as1}.\emph{\ref{nulo}}. Then $\lim_{t\to \infty} \psi(t r)\cdot r = |r|$ and so $|\psi(tr)\cdot r|\le |r|$ for any $t\ge 0$. Choosing $t=1$ we conclude the proof.
\end{proof}

A sufficient condition in order to have Assumptions \ref{as1}.\emph{\ref{parity}} when $d>1$ is to impose a certain form of isotropy. Let us state this as a  separate assumption. 
\begin{Assumptions}
\label{giso}
Let $d>1$ and assume that $\psi(r)= r g(|r|)$ for some function $g:\R_0^+\rightarrow \R^+$ such that $g \in C^1(\R_0^+)$ and $|r g'(r)/g(r)|\le 1$ for any $r \in \R^+_0$.
\end{Assumptions}
\begin{Remark}
\label{gprops}
{\rm
If Assumptions \ref{as1} and \ref{giso} are taken at the same time, then the function $g$ satisfies 
\begin{equation}
 \lim_{|r|\to \infty}|r|\, g(|r|)=1.\label{ginfinito}
\end{equation}
}
\end{Remark}

The following result shows our claim.

\begin{Lemma}
\label{herramienta}
Let Assumptions \ref{giso} hold true. Then Assumptions \ref{as1}.\ref{parity} are verified.
\end{Lemma}
\begin{proof}
To prove that $\psi$ is conservative, let $G$ be defined as $G(|r|^2):=g(|r|)$. If $\bar G$ is a primitive for $G$, then $\bar G(|r|^2)/2$ is a potential for $\psi(r)$. Assumption (b) follows immediately. To show (c) we use a variant of Sylvester's determinant theorem, stating that for an invertible $d\times d$ matrix $A$,
\begin{equation}
\label{trick}
\det (A + u v^T) = \det(A) \, (1+v^T A^{-1} u),\quad A \in M(\R^d,\R^d),\ u,v \in \R^d.
\end{equation}
We rest on Sylvester's criterion for quadratic forms as well. Under the present assumptions 
\begin{equation}
\label{parij}
\partial_i\psi^{(j)}= \delta_{ij} g(|r|) + g'(|r|) \frac{r_ir_j}{|r|},\quad i,j = 1,\ldots,d.
\end{equation}
Thus, the k-th principal minor of this matrix can be computed according to \eqref{trick}  as
$$
g(|r|)^k \left( 1+\frac{g'(|r|)}{g(r)} \frac{\sum_{i=1}^k r_i^2}{|r|}\right)
$$
and the result follows. As to (d), we resort to \eqref{parij} again. Thus, when $|r|\gg 1$,
$$
|r| \left|\partial_i \psi^{(j)}(r)\right| \le |r| g(|r|) + |r|^2 g'(|r|) \le 2 |r| g(|r|) \le 2.
$$
\end{proof}

The previous set of assumptions will allow to cover the case of \eqref{template}. In order to deal with \eqref{template2}, let us specify what do we demand of the function $\varphi$.
\begin{Assumptions}
\label{asphi}
Let $\varphi: \R \mapsto \R_0^+$ satisfy the following: 
\begin{enumerate}

\item $\varphi$ is Lipschitz continuous. \label{cnt}

\item $\varphi(0)=0$ and $\lim_{z\to 0} \varphi(z)/|z|=\varphi'(0)$ exists and is finite. \label{zzero}

\item  $\varphi(z)>0$ if $z\neq 0$.
\end{enumerate}
\end{Assumptions}
\begin{Remark}
{\rm
The particular case of \eqref{template} is covered here by the choice $\varphi(z)=s z$ as already noted before. Assumptions \ref{asphi} are automatically satisfied in this case.
}
\end{Remark}

\medskip
All considerations so far set up our framework. We point out first that under Assumptions \ref{as1} and \ref{asphi} equations of the form \eqref{template2} are well posed. This will be a consequence of the results in \cite{ACMEllipticFLDE, ACMMRelat}, as soon as we show that under the aforementioned set of assumptions equations of the form \eqref{template2} fall under their framework. We review the class of \emph{entropy solutions} and related notions like \emph{entropy conditions} and associated technicalities in Section \ref{ACM}; details on how to connect \eqref{template2} with that framework are given in Section \ref{wp}.

Taking this for granted, we are now ready to state the main results of the document. These collect several extensions of techniques and results in \cite{ARMAsupp, leysalto} together with some new ideas in order to describe several properties of the class \eqref{template2} and its subclass \eqref{template}. 
\begin{Theorem} 
\label{Main2}
Let $\psi$ verify Assumptions \ref{as1} and let $\varphi$ satisfy Assumption \ref{asphi}. Then the following assertions hold true:
\begin{enumerate}

\item (\emph{evolution of discontinuities}) \label{p3}
Let $u \in C([0,T];L^1(\R^d))$ be a distributional solution of \eqref{template2} with initial datum $0\le u_0 \in L^\infty(\R^d)\cap BV(\R^d)$. Assume that $u \in BV_{loc}(Q_T)$ and that the singular part of the spatial derivative has no Cantor part. Assume further that $\varphi$ is a convex function. Then the entropy conditions hold if and only if 
\begin{equation}
\nonumber
[\z \cdot \nu^{J_{u(t)}}]_+ = \varphi(u^+) \quad \mbox{and} \quad [\z \cdot \nu^{J_{u(t)}}]_- = \varphi(u^-).
\end{equation}
holds at each jump discontinuity (being $u^+ > u^-\ge 0$ the lateral traces of the solution and $[\z \cdot \nu^{J_{u(t)}}]_+,\,  [\z \cdot \nu^{J_{u(t)}}]_-$ the lateral traces of the flux). Moreover the speed of any discontinuity front is given by
\begin{equation}
\nonumber
\v=\frac{\varphi(u^+)-\varphi(u^-)}{u^+- u^-}.
\end{equation}

\item (\emph{evolution of the support} --see also \cite{Giacomelli}) 
\label{pp4}
Consider a compactly supported initial datum $0\le u_0 \in L^\infty(\R^d)$ and let $u(t)$ be the associated entropy solution. Then 
$$
\mbox{supp}\, u(t) \subset cl\left(\mbox{supp}\, u_0 \oplus B(0,\theta t)\right),\quad \theta= \max_{0\le z\le \|u_0\|_\infty} \varphi'(z)
$$
for every $t>0$.

\end{enumerate}
\end{Theorem}

There are some additional properties which are specific of \eqref{template}, as we state now.

\begin{Theorem} 
\label{Main}
Let $\psi$ verify Assumptions \ref{as1}. Then the following assertions hold true:
\begin{enumerate}

\item (\emph{cost functions with bounded domain}) \label{cfbd}
There exists a convex cost function $k:\R^d \rightarrow \R_0^+$ such that \eqref{template} can be (formally) recovered from the point of view of optimal mass transport problems associated with $k$ (as detailed in Section \ref{Ftransport}). Furthermore, $k$ is finite on $\{v\in \R^d/ |v|< s\}$ and assumes the value $+\infty$ on $\{v\in \R^d/ |v|> s\}$.

\item (\emph{evolution of the support and strict positivity}) 
\label{p5}
Consider a compactly supported initial datum $0\le u_0 \in L^\infty(\R^d)$  and let $u(t)$ be the associated entropy solution of \eqref{template}. If $u_0$ satisfies the positivity assumption \eqref{local_sep} and either $d=1$ or Assumptions \ref{giso} holds with $\lim_{r \to \infty} |r|^2 g'(r)=-1$, then
$$
\mbox{supp} \, u(t) = cl \left(\mbox{supp} \, u_0 \oplus B(0,st)\right) \ \forall t\ge 0
$$
and $u(t)$ is strictly positive inside its support for every $t>0$. 

\item (\emph{persistence of discontinuous interfaces in dimension one}) 
\label{p6}
Consider an initial datum $0\le u_0 \in L^1(\R)\cap L^\infty(\R)$ supported on a bounded interval $[a,b]$. Let $u(t)$ be the associated entropy solution of \eqref{template}. Assume that there exist some $\epsilon,\alpha>0$ such that $u_0(x)>\alpha$ for every $x\in (b-\epsilon,b)$. If there exist some $\tilde \epsilon>0$ such that
$$
  d(r) =O(1/r)\quad \mbox{and}\quad   \psi'(r) =O(r^{-2-\tilde \epsilon}) \quad \mbox{as}\quad r \to \infty,
$$
then the left lateral trace of $u(t)$ at $x=b+st$ is strictly positive for any $t>0$. A similar statement holds for the left end of the support.
\end{enumerate}
\end{Theorem}

We do not expect these results to generalize easily to the setting of Theorem \ref{Main2}. First, it does not seem to be possible to recast equations of the form \eqref{template2} as equations derived from an optimal mass transportation problem when $\varphi$ is not linear, not even allowing for any kind of convex entropy. On the contrary, a generalization of Theorem \ref{Main}.\emph{\ref{cfbd}} to some of the models presented in \cite{Vicent2013} seems feasible (see for instance \cite{2014}). As regards the evolution of the support, equations of the form \eqref{template2} are expected to display waiting time phenomena, a fact which has been already confirmed in some cases \cite{ACMSV,Giacomelli}. Hence to track the detailed evolution of the support for \eqref{template2} is outside the scope of the techniques we use here;  sub-solutions suited to this task must be able to capture what the waiting time for a given initial datum would be, which appears to be a very challenging problem.

It is also interesting to notice that the formal limit $L\to \infty$ turns \eqref{template} into a diffusion equation in transparent media (see \cite{TrMedia} and references therein),
\begin{equation}
\label{trmed}
\frac{\partial u}{\partial t} = \div \left( s\,u\,  \frac{\nabla u}{|\nabla u|}\right).
\end{equation}
We also note that performing the limit $s\to \infty,\, L \to 0$ we may arrive to linear diffusion equations. 
Namely, let $\nu:=\lim_{s \to \infty,\, L \to 0} sL$. Then, when $d=1$ we arrive to the following equation
$$
u_t=\psi'(0) \nu u_{xx}.
$$
If we assume a structure like that in Assumptions \ref{giso}, the limit equation in higher dimensions would be
$$
u_t= \nu g(0) \Delta u.
$$
A rigorous analysis of these limit cases, together with the study of regularity properties of solutions to \eqref{template}--\eqref{template2} will be the subject of future investigations.


\subsection{Examples}
We present here a non-comprehensive list of partial differential equations that are related to \eqref{template}--\eqref{template2}; most of them are already present in the literature. 

\begin{enumerate}

\item The standard heat equation \eqref{lheat} can be recast in the form \eqref{template} with $L=\nu/s$, $s=1$ and $\psi(r)=r$. It does not satisfy Assumptions \ref{as1}, though, as the velocity $V=\nu \nabla u/u$ is not bounded.

\item The porous media equations
$$
u_t = \nu\, \div ((u/\kappa)^{m-1} \nabla u),\quad m>1
$$
do not fall in the scope of \eqref{template}. They fit the form \eqref{Vform}, with velocity
$$
V=\nu \frac{u^{m-2}}{\kappa^{m-1}} \nabla u
$$ 
given by Darcy's law --but note that it is not bounded. Then they can be written as \eqref{template2} with $\varphi (u)= u^{m-2}/\kappa^{m-1}$ and $\psi(r)=r$, but this does not verify Assumptions \ref{as1}. 

\item None of Berstch--Dal Passo models \cite{BdPasso} falls in our framework. 

\item The relativistic heat equation \eqref{rhe} \cite{Rosenau92, Brenier1} has been already discussed in the introduction. So far it is the most popular model in the mathematical literature that fits into \eqref{template}.

\item Wilson's model was also mentioned in the introduction. It has the following form:
\begin{equation}
\label{Wilson}
\frac{\partial u}{\partial t} = \nu \, \div \left( \frac{|u| \nabla u}{|u| +\frac{\nu}{c} |\nabla u|}\right),\quad \nu,\, c>0.
\end{equation}
 This fits into \eqref{template} with
$$
V = \frac{\nu\, \sign (u) \nabla u}{|u| + \frac{\nu}{c}|\nabla u|}= c \frac{\frac{\nu}{c}\frac{\nabla u}{u}}{1 + \frac{\nu}{c}\left|\frac{\nabla u}{u}\right|}.
$$
Thus $s:=c$, $L:=\nu/c$ and $\psi(r):= r/(1+|r|)$. 

\item We can regard the relativistic heat equation and Wilson's model as particular instances of a one-parametric family of models that will be useful in order to probe a number of things in the sequel. Let us introduce a family of models depending on a parameter $p \in [1,\infty)$ by means of
\begin{equation}
\label{pmodelos}
\frac{\partial u}{\partial t} = \nu \, \div \left( \frac{|u| \nabla u}{\left(|u|^p +\frac{\nu^p}{c^p} |\nabla u|^p\right)^{1/p}}\right),\quad \nu,\, c>0.
\end{equation}
{This family seems to have been first introduced in the astrophysical literature by E. Larsen (see \cite{Olson})}. Here we have 
$$
\psi(r)= \frac{r}{(1+|r|^p)^{1/p}}.
$$
Note that for $p=2$ and $p=1$ we recover the relativistic heat equation and Wilson's model respectively. For any $p \in [1,\infty)$, \eqref{pmodelos} falls under the scope of Lemma \ref{herramienta} and Assumptions \ref{as1} are satisfied. The statement in point \emph{\ref{p5}} of Theorem \ref{Main} applies for every $p \in [1,\infty)$; the same happens for point \emph{\ref{p6}}, except for the case $p=1$ (Wilson's model).

\item {The following flux-saturated model was also introduced in the astrophysical literature} \cite{Levermore81}: 
$$
u_t = [s\, u (\coth(Lu_x/u) - u/(Lu_x))]_x.
$$
Here we have $\psi(r)= \coth (r) - 1/r$. Assumptions \ref{as1} are also satisfied in this case. Point \emph{\ref{p6}} in Theorem \ref{Main} does not apply for this model, though.

\item A general family of one-dimensional models having the form \eqref{template} was introduced in \cite{Coulombel} as part of a more general program concerning diffusive approximations of kinetic models via moment systems. For them $\psi$ is a $C^\infty$ function which is odd and strictly increasing. The simplest instance given in \cite{Coulombel} is: 
$$
 p_t = ((p/\epsilon) \tanh (\epsilon p_x/(\gamma p)))_x.
 $$
 In our notation, we have 
 $$
 V= \frac{1}{\epsilon} \tanh (\frac{\epsilon p_x}{\gamma p}).
 $$
This particular example satisfies Assumptions \ref{as1}; point \emph{\ref{p6}} in Theorem \ref{Main} applies in this case.

\item If we choose $\psi$ so that its Jacobian matrix is compactly supported, the model \eqref{template} agrees with \eqref{trmed} for large values of the ratio $|\nabla u/u|$.

\item Equations like \eqref{mrhe} \cite{ACMSV,leysalto} or more generally those of the form
\begin{equation}
\label{firhe}
\frac{\partial u}{\partial t} = \nu \, \div \left( \frac{\varphi(u) \nabla u}{\sqrt{u^2 +\frac{\nu^2}{c^2} |\nabla u|^2}}\right)
\end{equation}
 fall under the framework given by \eqref{template2}, provided that $\varphi$ satisfies Assumptions \ref{asphi}.

\item Models of the form 
$$
u_t = \nu \, \div \left(\frac{u \nabla u^m}{\sqrt{1+\frac{\nu^2}{c^2}|\nabla u^m|^2}}\right),\quad u_t = \alpha \, \div \left(\frac{\Lambda(u) \nabla \Phi(u)}{\sqrt{1+\beta |\nabla \Phi(u)|^2}}\right),
$$
which were introduced in \cite{Vicent2013}, do not fall into our framework, although some of their properties are quite similar. More precisely, \eqref{template2} agrees exactly with the second equation above for the choice $\Phi(u)=\log u$, but this choice is forbidden by the assumptions on $\Phi$ set in \cite{Vicent2013}. Hence the families of models treated in \cite{Vicent2013} and those discussed here are disjoint.

\end{enumerate}

\section{A summary about entropy solutions for degenerate parabolic equations with linear growth Lagrangian}
\label{ACM}

The class of equations given by \eqref{template2} is a subclass of the set of second order diffusion equations in divergence form
\begin{equation}
\label{DirichletproblemP}
u_t= \div \a(u,\nabla u) \quad \mbox{in}\quad Q_T
\end{equation}
 which have both a degeneracy with respect to $u$ (more precisely, $\lim_{z\to 0^+}\a(z,\xi)=0$ for any $\xi \in \R^d$) and a Lagrangian having linear growth at infinity, in the sense that
\begin{equation}
\label{deffi}
\frac{1}{|\xi|} \lim_{t \to +\infty} \a(z,t\xi)\, \xi = \varphi(z)
\end{equation}
for some function $\varphi$ and for every $z\ge 0$. This is roughly the class of equations for which a well-possedness theory was developed in the series of papers \cite{ACMEllipticFLDE, ACMMRelat}. The main tool there is the concept of \emph{entropy solution}, which was introduced in the previous papers and shown to provide a suitable class of solutions in which the well-posedness of the former problem is granted. This notion of solution is based on a set of Kruzkov's type inequalities and it requires to define a
functional calculus for functions whose truncations are of bounded variation. The purpose of this section is to collect a number of definitions and results (which we borrow from \cite{ACMEllipticFLDE, ACMMRelat, ARMAsupp, leysalto}) that are needed to work with such entropy solutions, which is the concept of solution that we will use in order to deal with our specific class of flux-saturated equations \eqref{template2} in the following sections. For that, we introduce functions of bounded variation, several classes of truncation functions, a suitable integration by parts formula, lower semicontinuity results for functionals defined on $BV$ and then the functional calculus itself. This allows to introduce the concept of entropy solutions and to state an existence and uniqueness result for such. We conclude the section with a comparison principle for sub- and super- solutions.

\subsection{Functions of bounded variation and some generalizations}
\label{sect:bv}

Denote by ${\mathcal L}^d$ and ${\mathcal H}^{d-1}$ the
$d$-dimensional Lebesgue measure and the $(d-1)$-dimensional
Hausdorff measure in $\R^d$, respectively. 

Recall that if $\Omega$ is an open subset of $\R^d$, a function $u
\in L^1(\Omega)$ whose gradient $Du$ in the sense of distributions
is a vector valued Radon measure with finite total variation in
$\Omega$ is called a function of bounded variation. The
class of such functions will be denoted by $BV(\Omega)$.  For $u
\in BV(\Omega)$, the vector measure $Du$ decomposes into its
absolutely continuous and singular parts $Du = D^{ac} u + D^s u$.
Then $D^{ac} u = \nabla u \ \L^d$, where $\nabla u$ is the
Radon--Nikodym derivative of the measure $Du$ with respect to the
Lebesgue measure $\L^d$. We also split $D^su$ in two parts: The jump part $D^j u$ and the Cantor part $D^c u$. 

We say that $x \in \Omega$ is an approximate jump point of $u$ if there exist $u^+(x) \neq u^-(x) \in \R$ and $\eta_u(x) \in  \mathbb{S}^{d-1}$ such that
$$
\lim_{\rho \searrow 0} \frac{1}{\L(B_\rho^+(x,\eta_u(x)))} \int_{B_\rho^+(x,\eta_u(x))} |u(y)-u^+(x)| \, dy = 0
$$
and
$$
\lim_{\rho \searrow 0} \frac{1}{\L(B_\rho^-(x,\eta_u(x)))} \int_{B_\rho^-(x,\eta_u(x))} |u(y)-u^-(x)| \, dy = 0,
$$
where
$$
B_\rho^+(x,\eta_u(x)) = \{y\in B(x,\rho)/ (y-x)\cdot \eta_u(x)>0\} 
$$
and
$$
B_\rho^-(x,\eta_u(x)) = \{y\in B(x,\rho)/ (y-x)\cdot \eta_u(x)<0\}.
$$
We denote by $J_u$ the set of approximate jump points. It is
well known (see for instance \cite{Ambrosio}) that 
$$
D^j u = (u^+
- u^-) \nu_u \H^{d-1} \res J_u,
$$ 
with $\nu_u(x) =\frac{Du}{\vert D u \vert}(x)$, 
being $\frac{Du}{\vert D u \vert}$
the Radon--Nikodym derivative of $Du$ with respect to its total
variation $\vert D u \vert$. For further information concerning
functions of bounded variation we refer to \cite{Ambrosio}.

\subsection{Several classes of truncation functions}
We will use in the sequel a number of different truncation functions. For $a <
b$ and $l \in \R$, let $T_{a,b}(r) := \max \{\min \{b,r\},a\}$, $ T_{a,b}^l =  T_{a,b}-l$.
We denote \cite{ACMEllipticFLDE,ACMMRelat,ARMAsupp}
\[
\begin{split}
\mathcal T_r & := \{ T_{a,b} \ : \ 0 < a < b \},  
\\
\mathcal{T}^+ & := \{ T_{a,b}^l \ : \ 0 < a < b ,\, l\in \R, \, T_{a,b}^l\geq 0 \},  
\\
\mathcal{T}^- & := \{ T_{a,b}^l \ : \ 0 < a < b ,\, l\in \R, \, T_{a,b}^l\leq 0 \}.  
\end{split}
\]
Given any function $w$ and $a,b\in\R$ we shall use the notation
$\{w\geq a\} = \{x\in \R^d: w(x)\geq a\}$, $\{a \leq w\leq b\} =
\{x\in \R^d: a \leq w(x)\leq b\}$, and similarly for the sets $\{w
> a\}$, $\{w \leq a\}$, $\{w < a\}$, etc.

We need to consider the following function space
$$
TBV_{\rm r}^+(\R^d):= \left\{ w \in L^1(\R^d)^+  \ :  \ \ T_{a,b}(w) - a \in BV(\R^d), \
\ \forall \ T_{a,b} \in \mathcal T_r \right\}.
$$
Using the chain rule for BV-functions (see for instance
\cite{Ambrosio}), one can give a sense to $\nabla u$ for a
function $u \in TBV^+(\R^d)$ as the unique function $v$ which
satisfies
\begin{equation*}\label{E1WRN}
\nabla T_{a,b}(u) = v \1_{\{a < u  < b\}} \ \ \ \ \ {\mathcal
L}^d-a.e., \ \ \forall \ T_{a,b} \in \mathcal{T}_r.
\end{equation*}
We refer to \cite{6son6} for details.

Let us denote by ${\mathcal P}$ the set of Lipschitz continuous
functions $p : [0, +\infty) \rightarrow \R$ satisfying
$p^{\prime}(s) = 0$ for $s$ large enough. We write ${\mathcal
P}^+:= \{ p \in {\mathcal P} \ : \ p \geq 0 \}$.

We define \cite{leysalto} $\mathcal{TSUB}$ as the class of functions $S,T \in \mathcal{P}$ such that 
$$
S\ge 0, \quad S'\ge 0 \quad \mbox{and} \quad T\ge 0, \quad T'\ge 0
$$
and $p(r)=\tilde{p}(T_{a,b}(r))$ for some $0<a<b$, being $\tilde{p}$ differentiable in a neighborhood of $[a,b]$ and $p=S,T$. Similarly, we introduce $\mathcal{TSUPER}$ as the class of functions $S,T \in \mathcal{P}$ such that 
$$
S\le 0, \quad S'\ge 0 \quad \mbox{and} \quad T\ge 0, \quad T'\le 0
$$
and $p(r)=\tilde{p}(T_{a,b}(r))$ for some $0<a<b$, being $\tilde{p}$ differentiable in a neighborhood of $[a,b]$ and $p=S,T$.

Finally, we let $J_q(r)$ denote the primitive of $q$ for any real function $q$; i.e. 
$$\displaystyle J_q(r):=\int_0^r q(s)\,ds.$$
\subsection{A generalized Green's formula}
Assume that $\Omega$ is an open bounded set of $\R^d$ with Lipschitz continuous boundary. Let $p\ge 1$ and $p'$ its dual exponent. Following \cite{Anzellotti1}, let us denote
\begin{equation}
\nonumber
X_{p}(\Omega) = \{ \z \in L^{\infty}(\Omega, \R^d):
\div(\z)\in L^p(\Omega) \}.
\end{equation}
If $\z \in X_{p}(\Omega)$ and $w \in BV(\Omega)\cap L^{p'}(\Omega) $,
we define the functional $(\z\cdot Dw): \mathcal{C}^{\infty}_{c}(\Omega)
\rightarrow \R$ by the formula
\begin{equation}
\nonumber
\langle (\z \cdot Dw),\varphi\rangle := - \int_{\Omega} w \, \varphi \,
\div(\z) \, dx - \int_{\Omega} w \, \z \cdot \nabla \varphi \, dx.
\end{equation}
Then $(\z \cdot Dw)$ is a Radon measure in $\Omega$ \cite{Anzellotti1}, and
\begin{equation}
\nonumber
\int_{\Omega} (\z \cdot Dw) = \int_{\Omega} \z \cdot \nabla w \, dx, \ \
\ \ \ \forall \ w \in W^{1,1}(\Omega) \cap L^{\infty}(\Omega).
\end{equation}
We denote by
$(\z \cdot Dw)^{ac}$, $(\z \cdot Dw)^{s}$ its absolutely continuous and singular parts with respect to $\mathcal{L}^d$.
One has that $(\z \cdot Dw)^{s}$ is absolutely continuous with respect to
$D^{s}w$ and $(\z \cdot Dw)^{ac}=\z\cdot \nabla w$. Moreover, $(\z \cdot Dw)$ is
absolutely continuous with respect to $\vert Dw \vert$ \cite{Anzellotti1}.

The weak trace on $\partial \Omega$ of the normal component of $\z \in X_p(\Omega)$ is defined in \cite{Anzellotti1}. More precisely, it is proved that there exists a linear operator $\gamma : X_p(\Omega)\rightarrow L^\infty (\partial \Omega) $ such that $\|\gamma(\z)\|_\infty\le \|\z\|_\infty$ and $\gamma(\z)(x) = \z(x) \cdot \nu^\Omega(x)$ for all $x \in \partial \Omega$ --being $\nu^\Omega(x)$ the normal vector at $x$ which points outwards--,  provided that $\z \in C^1(\bar \Omega,\R^d)$. We shall denote $\gamma(\z)(x)$ by  $[\z \cdot \nu^\Omega](x)$. Moreover, the following Green's formula, relating the function $[\z\cdot \nu^\Omega]$ and the measure $(\z \cdot Dw)$, for $\z\in X_p(\Omega)$ and $w\in BV(\Omega)\cap L^{p'}(\Omega)$, is proved in \cite{Anzellotti1}
$$
\int_\Omega w \, \div(\z)\, dx + \int_\Omega (\z \cdot Dw) = \int_{\partial \Omega} [\z \cdot \nu^\Omega]w \, d \mathcal{H}^{d-1}.
$$

\subsection{Functionals defined on BV}\label{sect:functionalcalculus}

In order to define the notion of entropy solutions of
(\ref{DirichletproblemP}) we
need a functional calculus defined on functions whose truncations
are in $BV$. For that we need to introduce some functionals defined on functions of bounded variation \cite{ACMEllipticFLDE,ACMMRelat}.

Let $\Omega$ be an open subset of $\R^d$. Let $g: \Omega \times \R
\times \R^d \rightarrow [0, \infty)$ be a Borel function such that
\begin{equation*}\label{LGRWTH}
C(x) \vert \xi \vert - D(x) \leq g(x, z, \xi)  \leq M'(x) + M
\vert \xi \vert
\end{equation*}
for any $(x, z, \xi) \in \Omega \times \R \times \R^d$, $\vert
z\vert \leq R$, and any $R>0$, where $M$ is a positive constant
and  $C,D,M' \geq 0$ are bounded Borel functions which may depend
on $R$. Assume that $C,D,M' \in L^1(\Omega)$. Following Dal Maso \cite{Dalmaso} we consider the
functional:
\begin{equation}
\nonumber
\begin{array}{ll}
\displaystyle {\mathcal R}_g(u) := \int_{\Omega} g(x,u(x), \nabla
u(x)) \, dx + \int_{\Omega} g^0 \left(x,
\tilde{u}(x),\nu_u(x) \right) \,  d\vert D^c
u \vert \\ \\
\qquad \qquad \displaystyle + \int_{J_u} \left(\int_{u_-(x)}^{u_+(x)}
g^0(x, s, \nu_u(x)) \, ds \right)\, d \H^{d-1}(x),
\end{array}
\end{equation}
for $u \in BV(\Omega) \cap L^\infty(\Omega)$, being $\tilde{u}$
the approximated limit of $u$ \cite{Ambrosio}. The recession
function $g^0$ of $g$ is defined by
\begin{equation}
\label{Asimptfunct}
 g^0(x, z, \xi) = \lim_{t \to 0^+} t\, g \left(x, z, \xi/t
 \right).
\end{equation}
It is convex and homogeneous of degree $1$ in $\xi$.

In case that $\Omega$ is a bounded set, and under standard
continuity and coercivity assumptions,  Dal Maso proved in
\cite{Dalmaso} that ${\mathcal R}_g(u)$ is $L^1$-lower semi-continuous
for $u \in BV(\Omega)$. A very general result about the
$L^1$-lower semi-continuity of ${\mathcal R}_g$ in $BV(\R^d)$ can be found on \cite{DCFV}.

Assume now that $g:\R\times \R^d \to [0, \infty)$ is a Borel function
such that
\begin{equation}
\nonumber
C \vert \xi \vert - D \leq g(z, \xi)  \leq M(1+ \vert \xi \vert)
\qquad \forall (z,\xi)\in \R^d, \, \vert z \vert \leq R,
\end{equation}
for any $R > 0$ and for some constants  $C,D,M \geq 0$ which may
depend on $R$. Assume also that
\begin{equation}
\nonumber
\1_{\{u\leq a\}} \left(g(u(x), 0) - g(a, 0)\right), \1_{\{u \geq b\}} \left(g(u(x),
0) - g(b, 0) \right) \in L^1(\R^d),
\end{equation}
for any $u\in L^1(\R^d)^+$. Let $u \in TBV_{\rm r}^+(\R^d)  \cap
L^\infty(\R^d)$  and $T = T_{a,b}-l\in {\mathcal T}^+ $. For each
$\phi\in \mathcal{C}_c(\R^d)$, $\phi \geq 0$, we define the Radon measure
$g(u, DT(u))$ by
\begin{equation}
\label{FUTab}
\begin{array}{c}
\displaystyle \langle g(u, DT(u)), \phi \rangle : = {\mathcal R}_{\phi g}(T_{a,b}(u))+
\displaystyle\int_{\{u \leq a\}} \phi(x)
\left( g(u(x), 0) - g(a, 0)\right) \, dx 
\\
 \displaystyle
\displaystyle  + \int_{\{u \geq b\}} \phi(x)
\left(g(u(x), 0) - g(b, 0) \right) \, dx.
\end{array}
\end{equation}
If $\phi\in \mathcal{C}_c(\R^d)$, we write $\phi = \phi^+ -
\phi^-$ 
and we define 
$$\langle g(u, DT(u)), \phi \rangle : =
\langle g(u, DT(u)), \phi^+ \rangle- \langle g(u, DT(u)), \phi^- \rangle.$$

Note that the following is shown in  \cite{DCFV}: if $g(z,\xi)$ is continuous in $(z,\xi)$, convex
in $\xi$ for any $z\in \R$, and $\phi \in \mathcal{C}^1(\R^d)^+$ has compact
support, then  $\langle g(u, DT(u)), \phi \rangle$ is lower
semi-continuous in $TBV^+(\R^d)$ with respect to the 
$L^1(\R^d)$-convergence.

We can now define the required functional calculus (see
\cite{ACMEllipticFLDE,ACMMRelat,leysalto}). Let $S \in  \mathcal{P}^+$, $T \in \mathcal{T}^+$.
We assume that
$u \in TBV_{\rm r}^+(\R^d) \cap L^\infty(\R^d)$ and 
$$
\1_{\{u\leq a\}} S(u)\left(g(u(x), 0) - g(a, 0)\right), \1_{\{u \geq b\}} S(u)\left(g(u(x),
0) - g(b,0) \right) \in L^1(\R^d).
$$
Then we define $g_S(u,DT(u))$ as the Radon
measure given by (\ref{FUTab}) with $g_{S}(z,\xi) = S(z)
g(z,\xi)$. 

Let us introduce $h: \R \times \R^d \rightarrow \R$ defined by
\begin{equation}
\label{hdef}
h(z,\xi):=\a(z,\xi) \xi,
\end{equation}
being $\a$ the flux in \eqref{DirichletproblemP}. Under suitable assumptions on $\a$, the measure $h_S(u,DT(u))$ will make sense according to the previous functional calculus; see Section \ref{wp}.

\subsection{Entropy solutions of the evolution problem}\label{sect:defESpp}

Let $L^1_{w}(0,T,BV(\R^d))$  be the space of weakly$^*$
measurable functions $w:[0,T] \to BV(\R^d)$ (i.e., $t \in [0,T]
\to \langle w(t),\phi \rangle$ is measurable for every $\phi$ in the predual
of $BV(\R^d)$) such that $\int_0^T \| w(t)\|_{BV} \, dt$ is finite.
Observe that, since $BV(\R^d)$ has a separable predual (see
\cite{Ambrosio}), it follows easily that the map $t \in [0,T]\to
\Vert w(t) \Vert_{BV}$ is measurable. By  $L^1_{loc, w}(0, T,
BV(\R^d))$ we denote the space of weakly$^*$ measurable functions
$w:[0,T] \to BV(\R^d)$ such that the map $t \in [0,T]\to \Vert
w(t) \Vert_{BV}$ is in $L^1_{loc}(0, T)$.

\begin{Definition} 
\label{def:espb}
Assume that $0 \le u_0 \in L^1(\R^d)\cap L^\infty(\R^d)$. A
measurable function $u: (0,T)\times \R^d \rightarrow \R$ is an
{\it entropy  solution}   of (\ref{DirichletproblemP}) in $Q_T$ if $u \in \mathcal{C}([0, T]; L^1(\R^d))$,
$T_{a,b}(u(\cdot)) - a \in L^1_{loc, w}(0, T, BV(\R^d))$ for all
$0 < a < b$, and
\begin{itemize}
\item[(i)] \ $u_t = {\rm div} \, \a(u(t), \nabla u(t))$ in $\mathcal{D}^\prime(Q_T)$,
\item[(ii)]  $u(0) = u_0$, and 
\item[(iii)] \ the following
inequality is satisfied
\begin{equation}
\begin{array}{c}
 \label{eineq}
\displaystyle \int_0^T\int_{\R^d} \phi
h_{S}(u,DT(u)) \, dt + \int_0^T\int_{\R^d} \phi h_{T}(u,DS(u)) \, dt
\\
\leq  \displaystyle\int_0^T\int_{\R^d} \Big\{ J_{TS}(u(t)) \phi^{\prime}(t) - \a(u(t), \nabla u(t)) \cdot \nabla \phi \
T(u(t)) S(u(t))\Big\} dxdt, 
\end{array}
\end{equation}
 for truncation functions $S,  T \in \mathcal{T}^+$, and any  smooth function $\phi$ of
 compact support, in particular  those  of the form $\phi(t,x) =
 \phi_1(t)\rho(x)$, $\phi_1\in {\mathcal D}(0,T)$, $\rho \in
 {\mathcal D}(\R^d)$.
\end{itemize}
\end{Definition}
\noindent
{This definition is a simplification of the original one in \cite{ACMMRelat}, see \cite{ACMregularity} for instance}. Note that the statements in this paragraph and the following one hold under a set of assumptions on $\a$ that are described in \cite{ACMEllipticFLDE,ACMMRelat,leysalto}, which we denote collectively by $({\rm H})$. We have the following existence and uniqueness result \cite{ACMMRelat}.

\begin{Theorem}
\label{EUTEparabolic}
 Let assumptions
$({\rm H})$ hold. Then, for any initial datum $0 \leq u_0 \in L^1(\R^d)\cap L^\infty(\R^d)$
there exists a unique entropy solution $u$ of
(\ref{DirichletproblemP}) in $Q_T $ for every $T
> 0$ such that $u(0) = u_0$.
Moreover, if $u(t)$,
$\overline{u}(t)$ are the entropy solutions corresponding to
initial data $u_0$, $\overline{u}_0 \in L^{1}(\R^d)^+$ respectively, then
\begin{equation}
\nonumber
\Vert (u(t) - \overline{u}(t))^+ \Vert_1 \leq
\Vert (u_0 - \overline{u}_0)^+ \Vert_1 \ \ \ \ \ \ {\rm for \
all} \ \ t \geq 0.
\end{equation}
\end{Theorem}

Existence of entropy solutions is proved by using Crandall-Liggett's scheme \cite{CrandallLiggett}
and uniqueness is proved using Kruzhkov's doubling variables technique \cite{Kruzhkov,Carrillo1}.

\subsection{Sub- and super-solutions}
In order to use the comparison principles introduced in \cite{ARMAsupp} a certain technical condition is required.
\begin{Assumptions}
\label{compas}
Let the function $h$ {defined by} (\ref{hdef}) satisfy
$$
h(z,\xi) \le M(z) |\xi|
$$
for some positive continuous function $M(z)$ and for any $(z,\xi)\in \R \times \R^d$.
\end{Assumptions}

\begin{Definition}{\cite{ARMAsupp}}
\label{subsuper}
A measurable function $u : (0,T) \times \R^d \rightarrow \R_0^+$ is an entropy sub- (resp. super-) solution of \eqref{DirichletproblemP} if $u \in C([0,T],L^1(\R^d))$, $\a(u,\nabla u)\in L^\infty(Q_T)$, $T_{a,b}(u) \in L_{loc,w}^1(0,T,BV(\R^d))$ for every $0<a<b$ and the following inequality is satisfied:
\begin{equation}
\begin{array}{c}
\displaystyle \int_0^T\int_{\R^d} \phi
h_{S}(u,DT(u)) \, dt + \int_0^T\int_{\R^d} \phi h_{T}(u,DS(u)) \, dt
\\
\geq  \displaystyle\int_0^T\int_{\R^d} \Big\{ J_{TS}(u(t)) \phi^{\prime}(t) - \a(u(t), \nabla u(t)) \cdot \nabla \phi \
T(u(t)) S(u(t))\Big\} dxdt,  
\label{seineq}
\end{array}
\end{equation}
(resp. with $\le$) 
for any $\phi \in \mathcal{D}(Q_T)^+$ and any truncations $T \in \mathcal{T}^+,\ S \in \mathcal{T}^-$.
\end{Definition}
This implies that 
\begin{equation}
\label{sinside}
u_t \le \div \a(u,\nabla u) \quad \mbox{in}\ \mathcal{D}'(Q_T)
\end{equation}
(resp. with $\ge$). The following comparison principle was shown in \cite{ARMAsupp}:
\begin{Theorem}
\label{theocomp}
Let assumptions (H) and Assumptions \ref{compas} hold. Given an entropy solution $u$ of \eqref{DirichletproblemP} corresponding to an initial datum $0\le u_0\in (L^\infty \cap L^1)(\R^d)$, the following statements hold true:
\begin{enumerate}
\item if $\overline{u}$ is a super-solution of \eqref{DirichletproblemP} such that $\overline{u}(t) \in BV(\R^d)$ for a.e. $t \in (0,T)$, then
$$
\|(u(t)-\overline{u}(t))^+\|_1 \le \|(u_0-\overline{u}(0))^+\|_1 \quad \forall t \in [0,T],
$$
\item if $\overline{u}$ is a sub-solution of \eqref{DirichletproblemP} such that $\overline{u}(t) \in BV(\R^d)$ for a.e. $t \in (0,T)$, then
$$
\|(\overline{u}(t)-u(t))^+\|_1 \le \|(\overline{u}(0)-u_0)^+\|_1 \quad \forall t \in [0,T].
$$
\end{enumerate}
\end{Theorem}
Some extensions of this result have been shown in \cite{Giacomelli}.

\section{Well-posedness of the given class of models}
\label{wp}
The goal of this paragraph is to show that the class of equations \eqref{template2} (and more specifically their flux $\a(z,\xi)$) satisfies the set of assumptions (H) given in \cite{ACMEllipticFLDE, ACMMRelat}. This will be the case provided that Assumptions \ref{as1} and \ref{asphi} below hold true. Hence Theorem \ref{EUTEparabolic} would apply, ensuring well-posedness for the class \eqref{template2}.

Comparing \eqref{template2} with \eqref{DirichletproblemP}, we let $\a:\R \times \R^d \rightarrow \R_0^+$ be defined as 
$$
\a(z,\xi):= \varphi(z)\, \psi(\xi/|z|)\quad \mbox{if} \ z\neq 0,\quad \a(z=0,\xi) := 0.
$$
It follows from the previous that 
\begin{equation}
\a(z,0)=0 \quad \mbox{for every}\ z \in \R.
\label{cero}
\end{equation} 
\begin{Lemma}
\label{Continuidad}
There holds that $\frac{\partial \a}{\partial \xi_i} \in C((\R \times \R^d)\backslash\{0,0\},\R^d)$ for each $i=1,\ldots,d$.
\end{Lemma}
\begin{proof}
This is clear except maybe at $z=0$. Note that
$$
\frac{\partial \a^{(j)}}{\partial \xi_i} (z,\xi)= \frac{\varphi(z)}{|z|} \partial_i \psi^{(j)}(\xi/|z|)\quad \mbox{if}\ z\neq 0.
$$
Then
$$
\lim_{z \to 0}\frac{\partial \a}{\partial \xi_i} =0\quad \mbox{for any}\, \xi\neq 0,
$$
 thanks to Assumptions \ref{asphi}.\emph{\ref{zzero}} and {Assumptions \ref{as1}.\emph{\ref{caida}}/\emph{\ref{1dcaida}}}.
\end{proof}
\begin{Remark}
{\rm Note that assumption (H2) in \cite{ACMEllipticFLDE, ACMMRelat} requires $\frac{\partial \a}{\partial \xi_i}$ to be continuous at every point of $\R \times \R^d$ in order to apply their well-posedness results. Nevertheless, it can be shown that the continuity proved in Lemma \ref{Continuidad} suffices.}
\end{Remark}
Now we look for some sort of potential function $f(z,\xi)$ such that $\nabla_\xi f(z,\xi)=\a(z,\xi)$. We introduce the Lagrangian 
$$
f(z,\xi):= |z| {\varphi}(z) \Phi(\xi/|z|) \quad \mbox{if} \ z\neq 0, \quad f(z=0,\xi) := 0,
$$
being $\Phi$ a potential for $\psi$ such that $\Phi(0)=0$. This potential is uniquely given by Poincare's Lemma:
$$
\Phi(r):=\int_{\gamma_r} \psi \, d\sigma
$$
with
\begin{equation}
\label{path}
\gamma_r: [0,1]\rightarrow \R^d,\quad \gamma_r(t) := tr\, \mbox{for any}\ r \in \R^d.  
\end{equation}
Thanks to Assumptions \ref{as1}.\emph{\ref{convex}}, $\Phi$ is convex. Hence the vector field $\xi \mapsto \psi(\xi)$ is monotone. Using Assumptions \ref{as1}.\emph{\ref{nulo}} (or equivalently Lemma \ref{tangential-radial}) we get that
\begin{equation}
\label{enhancedparity}
r \,\psi(r) \ge 0\quad \forall r \in \R^d.
\end{equation}
It follows easily that $\Phi \ge 0$.
Moreover $\xi \mapsto \Phi(\xi/|z|)$ is convex and so
\begin{equation}
\label{pr3}
\xi \mapsto f(z,\xi)\quad \mbox{ is convex}.
\end{equation}

\begin{Lemma}
 There holds that
 $
 \Phi(r)=|r| + o(|r|)$ for  $|r|\gg 1$.
 
\end{Lemma}
\begin{proof}
 Note that given $r\in \R^d$,
 $$
 \Phi(r)= \int_0^{1} \psi(t r) \cdot r \, dt = |r| + \int_0^1 \left(\psi(tr)- \frac{r}{|r|} \right)\cdot r \, dt.
 $$
 Clearly
 $$
 \left| \int_0^1 \left(\psi(tr)- \frac{r}{|r|} \right)\cdot r \, dt\right| \le \int_0^1 |r| d(t|r|)\, dt = \int_0^{|r|} d(\lambda) \, d\lambda
 $$
 and the result follows thanks to Assumptions \ref{as1}.\emph{\ref{infinito}}.
 \end{proof}
 
 In particular, given  $\xi \in \R^d$,
$$
\Phi( \xi/|z|) \sim  |\xi|/|z| \quad \mbox{for}\  |z|\ll 1.
$$
Thus, thanks to Assumptions \ref{asphi} we deduce that $\lim_{z \to 0} f(z,\xi)=0$. As a consequence, 
\begin{equation}
\label{pr1} f \in C(\R \times \R^d).
\end{equation}
In the same vein, we can compute (recall that $f^0$ is defined by \eqref{Asimptfunct})
\begin{equation}
\label{pr2}
f^0(z,\xi) = \lim_{t \to 0^+} t f(z,\xi/t) = \lim_{t\to 0^+} t|z|{\varphi}(z) \Phi(\xi/|t z|) ={\varphi}(z) |\xi|.
\end{equation}

A number of bounds hold for the Lagrangian. The following upper bound is easily obtained {upon using Lemma \ref{tangential-radial}.\emph{\ref{dos}}}:
\begin{equation}
\label{pr4}
f(z,\xi) = |z|{\varphi}(z)\int_0^1 \psi(t \xi/|z|) \cdot \xi/|z| \, dt\le  |z|\varphi(z) \int_0^1{|\xi|/|z|} \,dt  \le \varphi (z) (1+|\xi|).
\end{equation}
\begin{Lemma}
There exist suitable constants $C_0,D_0>0$ such that the Lagrangian satisfies the following lower bound
\begin{equation}
\label{pr5}
f(z,\xi) \ge C_0 \varphi(z) |\xi| - D_0|z| \varphi(z) \quad \mbox{for any}\ (z,\xi)\in \R \times \R^d.
\end{equation}
\end{Lemma}
\begin{proof}
Thanks to the convexity of $\xi \mapsto \Phi(\xi/|z|)$ we get the following inequality
$$
\Phi(\xi/|z|) \ge  \frac{\xi}{|z|} \psi( \xi/|z|) \quad \mbox{for any}\ (z,\xi)\in \R \times \R^d.
$$
Then we show that
$$
\varphi(z) \psi(\xi/|z|)\xi \ge C_0 \varphi(z) |\xi| - D_0|z| \varphi(z) \quad \mbox{for any}\ (z,\xi)\in \R \times \R^d
$$
and for some $C_0,D_0>0$, which would lead us to \eqref{pr5}. For that we extend a one-dimensional argument presented in \cite{2014}. Choose $0<C_0<1$. Note that
$$
r \psi(r) = |r| + r \left(\psi(r) - \frac{r}{|r|} \right)\ge |r|-|r| d(|r|).
$$
Hence, there exists some $\tilde r $ depending on $C_0$ such that
$$
r \psi(r) \ge C_0 |r|\quad \forall r\in \R^d\backslash B(0,\tilde r).
$$
Next, thanks to \eqref{enhancedparity} we are able to find $D_0:=C_0 \tilde r>0$ such that
$$
r \psi(r) \ge C_0 |r| -D_0\quad \forall r \in \R^d.
$$
Then we choose $r=\xi/|z|$ above and multiply both sides of the resulting inequality by $|z|\varphi(z)$ to get the desired estimate.
\end{proof}

Next we introduce $h:\R \times \R^d \rightarrow \R$ defined by \eqref{hdef} as 
\begin{equation}
\label{hours}
h(z,\xi):=\a(z,\xi) \xi = {\varphi}(z)\psi \left(\xi/|z|\right) \xi \quad \mbox{if} \ z\neq 0,\quad h(0,\xi):=0.
\end{equation}
It follows from \eqref{enhancedparity} and  Assumptions \ref{as1}.\emph{\ref{minuse}} that
\begin{equation}
\label{pr6}
h(z,\xi)\ge 0
\end{equation} 
and 
\begin{equation}
\label{pr7}
h(z,\xi)=h(z,-\xi)
\end{equation}
 for every $z,\xi \in \R$. Moreover, $h^0$ exists and coincides with $f^0$:
\begin{equation}
\label{pr8}
h^0(z,\xi)= \lim_{t \to 0^+} t {\varphi}(z) \psi \left(\frac{\xi}{|t z|}\right)\frac{\xi}{t} = \lim_{t \to 0^+}  {\varphi}(z) \xi \frac{\xi/|tz|}{|\xi/|tz||} =  {\varphi}(z) |\xi|.
\end{equation}

We have the following inequality relating $h^0$ and $\a$:
\begin{equation}
\label{pr9}
\a(z,\xi) \eta \le h^0(z,\eta)\quad \mbox{for every}\ \xi,\eta \in \R^d\ \mbox{and}\ z \in \R.
\end{equation}
\begin{Lemma}
There holds that
\begin{equation}
\label{pr10}
\left| (\a(z,\xi) - \a(\hat z,\xi)) (\xi -\hat \xi)\right| \le C|z - \hat z| \, |\xi -\hat \xi|
\end{equation}
for any $(z,\xi),(\hat z,\hat \xi) \in \R \times \R^d$ and for some constant $C>0$ depending on $|z|,|\hat z|$.
\end{Lemma}
\begin{proof}
 Let us write
$$
\left| \a(z,\xi) - \a(\hat z,\xi)\right| \le |\psi(\xi/|\hat z|)|\, |\varphi(\hat z)-\varphi(z)| + \varphi(z) \left| \psi(\xi/|\hat z|)-\psi(\xi/|z|)\right| : = A+B.
$$
Recall that $\psi$ and $\varphi$ are locally Lipschitz. As $\|\psi \|_{L^\infty(\R^d,\R^d)}<\infty$ thanks to Assumptions \ref{as1}.\emph{\ref{suave}} and \ref{as1}.\emph{\ref{infinito}}, term $A$ is fine. 

To deal with $B$, let us consider first that $|z-\hat z|\ge 1$. Then it suffices to show that $B$ is bounded by a constant, uniformly in $\xi$ and locally in $z,\hat z$. This is easily shown to be the case due to the boundedness of $\psi$ and the local boundedness of $\varphi$. 

Let us treat now the case $|z-\hat z|<1$. Without loss of generality, assume that $|\hat z|\ge |z|$. Using the mean value theorem,
$$
\psi^{(j)}(\xi/|\hat z|) = \psi^{(j)}(\xi/|z|) + \nabla \psi^{(j)}(\theta_j) (\xi/|\hat z|-\xi/|z|)
$$
for some $\theta_j$ lying in the segment joining $\xi/|z|$ and $\xi/|\hat z|$, $j=1,\ldots, d$. Thus,
$$
B\le \frac{\varphi(z)}{|z|} \frac{|\xi|}{|\hat z|}  |z-\hat z| \Theta,\quad \Theta:= \sup_{\tiny{
\begin{array}{c}
0\le \lambda \le 1, 
\\
i,j\in \{1,\ldots,d\}
\end{array}
}
}
|\partial_i \psi^{(j)}(\lambda |\xi|/|z|+ (1-\lambda)|\xi|/|\hat z|)|.
$$
Being $\varphi(z)/|z|$ locally bounded, it suffices to bound $|\xi|\Theta /|\hat z|$ independently of $\xi$ and locally in $\hat z$. Invoking {Assumptions \ref{as1}.\emph{\ref{caida}}/\emph{\ref{1dcaida}}}, there are values $c,\tilde r>0$ such that $\|D\psi\|_\infty(r)\le c/|r|$ for any $r\in \R^d\backslash B(0,\tilde r) $. Thus, whenever $\tilde r <|\xi/\hat z|\le |\xi/z|$,
$$
\frac{|\xi|}{|\hat z|} \Theta \le c  \frac{|\xi|}{|\hat z|} \max \left\{\frac{|z|}{|\xi|},\frac{|\hat z|}{|\xi|} \right\}\le c.
$$
 If $|\xi/\hat z|<\tilde r$ we are also done as the entries of $D \psi$ are bounded. 
\end{proof}

Thanks to \eqref{cero} and \eqref{pr3}--\eqref{pr10} we can apply the well-posedness theory given in \cite{ACMMRelat} (more precisely Theorem \ref{EUTEparabolic} above). We get the following result.
\begin{Theorem}
\label{Existence_is_not_straightforward}
Consider an initial datum $0\le u_0\in L^1(\R^d)\cap L^\infty(\R^d)$. Let Assumptions \ref{as1} be fulfilled. Then the  following assertions hold true:
\begin{enumerate} 
\item There exists a unique entropy solution $u$ of \eqref{template} in $Q_T$ for every $T>0$ with $u_0$ as initial datum. 
\item Let $\varphi$ satisfy Assumption \ref{asphi}. Then there exists a unique entropy solution $u$ of \eqref{template2} in $Q_T$ for every $T>0$, such that $u(0)=u_0$. 
\end{enumerate}
Moreover, if we are given $u,\hat u$ two entropy solutions of \eqref{template}(resp.\eqref{template2}) corresponding to initial data $0\le u_0,\hat u_0 \in L^1(\R^d)\cap L^\infty(\R^d)$ respectively, then
$$
\|(u(t) -\hat u(t))^+Ê\|_1 \le \|(u_0 - \hat u_0)^+Ê\|_1\quad \forall t>0. 
$$ 
\end{Theorem}
Furthermore, using \eqref{hours} we notice at once that Assumptions \ref{compas} is satisfied too. Hence Theorem \ref{theocomp} holds under Assumptions \ref{as1} and \ref{asphi}.

\section{A connection with optimal transport theory}
\label{Ftransport}
The use of optimal mass transport problems to solve parabolic equations was pioneered by \cite{Jordan} and further developed by many authors, see \cite{Agueh, AmbrosioII, Brenier1} for instance. We give here a brief account on it. Let $k : \R^d \rightarrow [0,\infty]$ be a convex cost function and let us define the associated Wasserstein distance between two probability distributions $\rho_0$ and $\rho_1$ by
$$
W_k^h(\rho_0,\rho_1):= \inf \left\{\int_{\R^d \times \R^d} k \left(\frac{x-y}{h}\right)\ d\gamma(x,y) \bigg/ \gamma \in \Gamma(\rho_0,\rho_1) \right\},
$$
being $h>0$. Here $\Gamma(\rho_0,\rho_1)$ stands for the set of probability measures in $\R^d \times \R^d$ whose marginals are $\rho_0$ and $\rho_1$.

Now let $F :[0,\infty) \rightarrow [0,\infty)$ be a convex function and let $\mathcal{P}(\R^d)$ be the set of probability density functions $\rho : \R^d \rightarrow [0,\infty)$. Starting from $\rho_0^h =\rho_0 \in \mathcal{P}(\R^d)$, we can solve iteratively  
$$
\inf_{\rho \in \mathcal{P}(\R^d)} h W_k^h(\rho_{n-1}^h,\rho) + \int_{\R^d} F(\rho(x))\ dx.
$$
Define $\rho^h(t)=\rho_n^h$ for $t \in [nh,(n+1)h)$. Then as $h \to 0^+$ the solution of this minimization scheme formally converges to a limit $u$ which solves the following equation
\begin{equation}
\nonumber
u_t = \div\, (u \nabla k^* (\nabla F'(u))). 
\end{equation}
This convergence has been shown to be rigorous in certain cases \cite{Jordan,Agueh,McCann}. In particular, the relativistic heat equation \eqref{rhe} falls under this general picture for the choice
$$
F(r)=\nu (r\log r - r),
$$
with the following cost function:
\begin{equation}
\nonumber
k(v)=\left\{
\begin{array}{ll}
\left(1-\sqrt{1-|v|^2/c^2} \right) c^2 & \mbox{if}\ |v|\le c
\\  \\
+\infty & \mbox{if}\ |v|>c,
\end{array}
\right.
\end{equation}
so that
$$
k^*(v) = c^2 \left( \sqrt{1+|v|^2/c^2}-1\right)\quad \mbox{and}\quad  
\nabla k^*(v)= \frac{v}{\sqrt{1+|v|^2/c^2}}.
$$
This was observed in \cite{Brenier1} at a formal level and later made rigorous in \cite{McCann}.

We notice that this is no particular phenomenon: Equations coming from such minimization schemes may have the form \eqref{template}. Our main concern in this section is the following: If such a model verifies Assumptions \ref{as1}, what can be said about the cost function $k$?

To be able to compare both frameworks we must set $F(r)=L (r\log r - r)$, which would yield an equation of the following form:
\begin{equation}
\label{templatetr2}
u_t = \div\, (u \nabla k^* (L \nabla u/u)). 
\end{equation}
Then, the following result provides an answer to the previous question; we get a new way to describe the role of the constant $s$.
\begin{Proposition}
Let $\psi$ satisfy Assumptions \ref{as1}. Then there exists a convex cost function $k:\R^d \rightarrow \R_0^+$ such that \eqref{template} can be recast as \eqref{templatetr2}. Furthermore, $k$ is finite on $\{v\in \R^d/ |v|< s\}$ and assumes the value $+\infty$ on $\{v\in \R^d/ |v|> s\}$. Provided that the function $d$ in Assumptions \ref{as1}.\ref{infinito} is integrable, $k$ is also finite on $\{v\in \R^d/ |v|=s\}$.
\end{Proposition}
\begin{proof}
Comparing \eqref{template} with \eqref{templatetr2} we identify
$$
 s \psi(r) = \nabla k^*(r)\quad \forall r \in \R^d.
$$
We can construct $k^*$ in a way that it satisfies $k^*(0)=0$ (as we did with $\Phi$ in Section \ref{wp}). We set
$$
k^*(r):=s \int_{\gamma_r} \psi \, d\sigma,\quad  \gamma_r\ \mbox{as in}\ \eqref{path}.  
$$
As previously argued, $k^*$ so defined is non-negative, convex and regular enough so that Fenchel--Moreau's theorem applies. Then we have the following representation formula for $k:\R^d\rightarrow \R$:
$$
k(v)= \sup_{p \in \R^d} p v - k^*(p):=\sup_{p \in \R^d} \Gamma_v(p).
$$
Let us address the properties of $k$. We will make repeated use of Lemma \ref{tangential-radial} in the sequel. In order to compute the value of $k(v)$, let us note that $k^*(r) \nearrow s |r|$ for $|r|\to \infty$ irrespective of the direction. Thus, for $|p|$ large enough,
\begin{equation}
\label{gama_asint}
\Gamma_v(p) \sim |p| (|v|\cos \theta(p,v)-s),
\end{equation}
being $\theta(p,v)$ the angle formed by $p$ and $v$. So, whenever $|v|>s$, $\Gamma_v(p)$ diverges to $+\infty$ as a function of $p$ along the ray given by the direction of $v$. Hence $k(v)=+\infty$ for $v\in \R^d$ such that $|v|>s$.

Let us deal now with $k(v)$ when $|v|<s$. In this case we notice that, thanks to \eqref{gama_asint}, $\Gamma_v(p)$ diverges to $-\infty$ along any ray as a function of $p$. Then $\sup_{p \in \R^d} \Gamma_v(p)$ is attained at those $\bar p \in \R^d$ such that $\nabla_p \Gamma_v(\bar p) = 0$. We are led to solve
\begin{equation}
\label{localgama}
v = \nabla k^*(\bar p), \quad \mbox{that is}\quad \frac{v}{s}=  \psi(\bar p).
\end{equation}
Let us write $ \psi^{-1}(v/s)$ for the solution set of \eqref{localgama} (if $\xi \mapsto \psi(\xi)$ is strictly monotone we have a unique solution), so that 
$$
k(v) = \sup_{\bar p\in \psi^{-1}(v/s)} \bar p v -k^*(\bar p).
$$
This supremum is clearly finite. On the other hand, note that in dimension one we have
$$
k(v)=\sup_{\bar p\in \psi^{-1}(v/s)} \int_0^{\bar p}v-s \psi(\lambda) \ d\lambda \ge \sup_{\bar p\in \psi^{-1}(v/s)} \int_0^{\bar p}v-s \psi(\bar p) \ d\lambda =0
$$
as $\psi$ is non-decreasing. For higher dimensions, we show that the cost function-to-be is non-negative as follows:
$$
k(v) =\sup_{\bar p\in \psi^{-1}(v/s)} v \bar p - s \int_0^1\psi(t \bar p) \bar p\ dt
 \ge \sup_{\bar p\in \psi^{-1}(v/s)} v \bar p - s \int_0^1\frac{v}{s} \bar p\ dt =0.
$$
Hence $k(v)$ qualifies as cost function for $|v|<s$.

Finally we study the behavior of $k(v)$ when $|v|=s$. We start doing this in dimension one. First, we note that
$$
\Gamma_s(p)= s \int_0^p 1-\psi(\lambda) \ d\lambda, \quad \mbox{then}\quad 
\frac{d}{dp}\Gamma_s(p)= s (1-\psi(p))\ge 0.
$$
Then we compute the supremum taking the limit $p \to +\infty$ (this makes sense even when $\psi'$ is compactly supported). Thus
$$
k(s)= s \int_0^\infty 1- \psi(\lambda)\ d\lambda.
$$
Arguing in a similar way, $k(-s)$ is found to have the same value. Let us discuss now the higher dimensional case. Nothing precludes that the solution set $\psi^{-1}(v/s)$ of \eqref{localgama} be non-empty even for $|v|=s$. This is not troublesome as long as this set is bounded, as the associated contributions $\Gamma_v(p)$ to the value of $k(v)$ would be clearly bounded. Then let us discuss what happens for $|p|\to \infty$. According to \eqref{gama_asint}, $\Gamma_v(p)$ diverges to $-\infty$ along any ray except maybe along the ray determined by $v$ itself. In fact, let $p=\lambda \frac{v}{|v|}$ for $\lambda>0$ and compute for $|v|=s$
$$
\frac{d}{d\lambda} \Gamma_v\left(\lambda \frac{v}{|v|}\right) = \frac{d}{d\lambda} \left(\lambda s - s \int_0^\lambda \psi\left(t \frac{v}{|v|}\right)\frac{v}{|v|}\ dt\right) = s - v \psi \left(\frac{\lambda v}{s} \right) \ge 0.
$$
Hence,
$$
\lim_{\lambda \to +\infty} \Gamma_v\left(\lambda \frac{v}{|v|}\right)= \lim_{\lambda \to +\infty} s \int_0^\lambda 1- \frac{v}{|v|}\cdot \psi\left(t \frac{v}{|v|}\right) \ dt= s \int_0^\infty 1- \frac{v}{|v|}\cdot \psi\left(t \frac{v}{|v|}\right) \ dt.
$$
We conclude the proof by noticing that the integrability of $d$ in Assumptions \ref{as1}.\emph{\ref{infinito}} ensures the convergence of these improper integrals.
\end{proof}
Let us stress that there is at least a certain subclass of the class of functions $\psi$ satisfying Assumptions \ref{as1} such that the minimization procedure sketched at the beginning of the section produces actual solutions of \eqref{template}. See \cite{McCann} for details.


\section{Propagation of the support}
\label{spread}
The aim of this section is to supply proofs for points \emph{\ref{p5}} and \emph{\ref{p6}} in Theorem \ref{Main} and point \emph{\ref{pp4}} in Theorem \ref{Main2}. This is done by means of comparison with suitable sub- and super-solutions, in the same vein as \cite{ARMAsupp}. For that we will rest in Theorem \ref{theocomp}, which applies under Assumptions \ref{as1} and \ref{asphi} as argued in Section \ref{wp}.

\subsection{Upper bounds on support spreading rates}
\label{supers}
We show in this paragraph that dilations of multiples of characteristic functions of compact sets qualify as super-solutions if their spreading rate behaves in a suitable way. This is an extension of Proposition 1 in \cite{ARMAsupp}. 
{A generalization of the results in this paragraph has been independently discovered in \cite{Giacomelli}.}

\begin{Proposition}
\label{p62}
Let $\beta>0$ and $C\subset \R^d$ a compact set. Let Assumptions \ref{as1} and \ref{asphi} be satisfied. Then
$$
u(t,x) = \beta \chi_{B(t)},\quad \mbox{being}\ B(t):=C\oplus B(0,\theta t) \quad \mbox{with}\ \theta = \max_{0\le z\le \beta} \varphi'(z), 
$$
is a super-solution of \eqref{template2} in $Q_T$ for every $T>0$.
\end{Proposition}
\begin{proof}
We start by defining $\bar B(t):=C \oplus B(0,C(t))$ for some function $C(t)\ge 0$ with $C(0)=0$ and $C'(t)\ge 0$. Let us introduce now
$$
W(t,x) := \beta \chi_{\bar B(t)}.
$$
Fix $T>0$. We shall determine what extra conditions have to be imposed on $C(t)$ in order that $W$ be a super-solution of \eqref{template2} in $Q_T$. Note that $\a(W,\nabla W)=0$ for such a profile, {thanks to} \eqref{cero}.
As
$$
W_t = \beta C'(t) \H^{d-1}_{|\partial \bar B(t)},
$$
we get at once that 
\begin{equation}
\label{paso1}
 W_t \ge \div \a(W,\nabla W) \quad \mbox{in}\ \mathcal{D}'(Q_T).
\end{equation}

Next we compute each term in \eqref{seineq} of Definition \ref{subsuper} separately. Let $T \in \mathcal{T}^+$ and $S \in \mathcal{T}^-$.
Arguing as in \cite{ARMAsupp}, Proposition 1, we get that
\begin{equation}
\label{paso2.1}
 h_S(W(t),DT(W(t)))^s+h_T(W(t),DS(W(t)))^s = J_{(TS)' \varphi} (\beta) \H^{d-1}_{|\partial \bar B(t)}.
\end{equation}
Note that 
\begin{equation}
\begin{array}{ll}
\label{paso2.2}
\displaystyle J_{(TS)' \varphi} (\beta) = \int_0^\beta (TS)'(r) \varphi(r)\ dr & \displaystyle = -\int_0^\beta T(r)S(r) \varphi'(r) \ dr + T(\beta)S(\beta) \varphi(\beta)  \\ \\
&  = -J_{TS\varphi'}(\beta) + (TS\varphi)(\beta).
\end{array}
\end{equation}
Here we used that $\varphi(0)=0$. Apart from this, we notice that
$$
J_{TS}(W(t))= J_{TS}(\beta) \chi_{\bar B(t)}
$$
and so
$$
\partial_t J_{TS}(W(t))= C'(t) J_{TS}(\beta)  \H^{d-1}_{|\partial \bar B(t)}.
$$
Given any $0\le \phi \in \mathcal{D}'(Q_T)$, we have shown that
\begin{equation}
\label{paso3}
 \int_{Q_T} J_{TS}(W(t)) \phi'(t)\ dt  =  -\int_0^T \left\{C'(t)J_{TS}(\beta) \int_{\partial \bar B(t)} \phi \ d \H^{d-1} \right\}\ dt.
\end{equation}
Collecting \eqref{paso1}--\eqref{paso3} and comparing with inequality \eqref{seineq}, we will be done if we can show that the following inequality holds for any $T \in \mathcal{T}^+$, $S \in \mathcal{T}^-$ and $0\le \phi \in \mathcal{D}'(Q_T)$:
$$
\displaystyle \int_0^T \left\{[(TS\varphi)(\beta)- J_{TS\varphi'}(\beta)+ C'(t)J_{TS}(\beta)]Ê\int_{\partial \bar B(t)} \phi \ d \H^{d-1} \right\} \ dt \le 0.
$$
Here we have that $(TS\varphi)(\beta)\le 0$ as $S\le 0$. Note also that
$$
J_{TS\varphi'}(\beta)= \int_0 ^\beta T(r)S(r)\varphi'(r)\, dr \ge \theta \int_0 ^\beta T(r)S(r)\, dr=\theta J_{TS}(\beta)
$$
for $\theta = \max_{0\le z\le \beta} \varphi'(z)$. Thus, in order for $W$ to be a super-solution it is enough to ask for $\min_{t \in [0,T]} C'(t) \ge \theta$. This implies our result.
\end{proof}

\begin{Remark}
{\rm
Tracking the above proof we notice that we do not need a flux with structure as in \eqref{template2} in order that the argument works. The main requirement in order that the above proof goes through while Theorem \ref{theocomp} applies is that the flux must be such $\varphi$ can be defined by means of \eqref{deffi}, being $\varphi$ a Lipschitz-continuous function such that $\varphi(0)=0,\ \varphi(z)>0$ for $z\neq 0$ and $\varphi'(0)$ exists. This generic point of view is the one that is adopted in \cite{Giacomelli}. 
}
\end{Remark}
\begin{Corollary}
\label{c61}
Let Assumptions \ref{as1} and \ref{asphi} be verified. Let $0\le u_0\in L^1(\R^d)\cap L^\infty(\R^d)$ be compactly supported and let $u(t)$ the entropy solution of \eqref{template2} with $u_0$ as initial datum. Then 
$$
\mbox{supp}\, u(t) \subset cl \left(\mbox{supp}\, u_0 \oplus B(0,\theta t)\right),\quad \theta= \max_{0\le z\le \|u_0\|_\infty} \varphi'(z)
$$
for every $t>0$.
\end{Corollary}
%

\subsection{Lower bounds on support spreading rates}
To give lower bounds for the spreading rate of solutions to \eqref{template} we shall look for compactly supported sub-solutions. The following result will be helpful in so doing. It is inspired in the proof of Proposition 2 of \cite{ARMAsupp}.

\begin{Proposition}
\label{equi}
Let $W(t,x)$ such that $W(0,\cdot)$ is compactly supported and assume that $B=B(t):=\mbox{supp}\, W(t,\cdot) = \mbox{supp}\, W(0,\cdot) \oplus B(0,C(t))$, also satisfying
$W(t,\cdot)|_{\partial B}= \gamma(t)\ge 0$ and the regularity requirements set  in Definition \ref{subsuper}. Let $\varphi$ be defined by \eqref{deffi}. Given $T>0$, assume either:
\begin{enumerate}

\item Relation \eqref{sinside} holds inside the support.

\item $\gamma(t)=0$ for every $0\le t\le T$.

\end{enumerate}
or
\begin{enumerate}

\item Relation \eqref{sinside} holds inside the support.

\item $\gamma(t)>0$ for every $0\le t\le T$. 

\item $\sup_{t\in (0,T)} C'(t)\le \inf_z \varphi'(z)$.  \label{tezeta}

\item $[\a(W,\nabla W)\cdot\nu^B]=-\varphi(\gamma(t))$ for every $0\le t\le T$. \label{tangencia}

\end{enumerate}
Then, $W$ fulfills \eqref{seineq} in Definition \ref{subsuper}.
\end{Proposition}
\begin{Remark}
{\rm
Informally speaking, condition \emph{\ref{tangencia}} above means that the profile is concave in a neighborhood of the interface and the contact angle is vertical. See Remark \ref{contacto} in that regard.
}
\end{Remark}
\begin{proof}
 Let $0 \le \phi \in \mathcal{D}(Q_T)$, $T \in \mathcal{T}^+$ and $S\in \mathcal{T}^-$. We compute each term in \eqref{seineq} of Definition \ref{subsuper} separately. First, arguing as in the proof of Proposition \ref{p62}, 
\begin{equation}
\label{t1}
h_T(W,DS(W))^s + h_S(W,DT(W))^s = \left[(TS\varphi)(\gamma(t)) - J_{TS \varphi'}(\gamma(t)) \right] \H_{| \partial B}^{d-1}.
\end{equation}
We compute also
\begin{equation}
\label{t2}
\partial_t J_{TS}(W) = W_t T(W) S(W) \chi_B + C'(t) J_{TS}(\gamma(t)) \H_{| \partial B}^{d-1}.
\end{equation}
Moreover, letting $\z = \a(W,\nabla W)$, 
\begin{equation}
\label{t3}
\begin{array}{ll}
\displaystyle \int_{Q_T} \z \nabla \phi T(W)S(W) \ dxdt = & \displaystyle -\int_{Q_T}  \phi\,  \div (\z T(W) S(W)) \ dxdt
\\ \\
& \displaystyle +\int_0^T \int_{\partial B} [\z T(W)S(W)\cdot \nu^B]\phi\,  d \H^{d-1}\ dt.
\end{array}
\end{equation}
Collecting \eqref{t1}--\eqref{t3}, \eqref{seineq} reads now as follows:
\[
\begin{split}
 \int_0^T & \left[(TS\varphi)(\gamma(t)) - J_{TS \varphi'}(\gamma(t)) \right] \int_{\partial B} \phi(t) d\H^{d-1}\ dt
\\
&+\int_{Q_T} \left[h_S(W,DT(W))^{ac} + h_T(W,DS(W))^{ac} Ê\right] \phi \ dt
\\
\ge & \int_{Q_T}  \phi\,  \div (\z T(W) S(W)) \ dxdt - \int_0^T \int_B \phi W_t T(W) S(W) \ dxdt
\\
& -\int_0^T \int_B [\z T(W)S(W)\cdot \nu^B]\phi \, d \H^{d-1}\ dt
 - \int_0^T J_{TS}(\gamma(t)) C'(t) \int_{\partial B} \phi(t) d \H^{d-1}\ dt.
\end{split}
\]
Our aim is to show that this holds true indeed. It is equivalent to check the above inequality for the absolutely continuous and singular parts separately.  As
\[
\begin{split}
h_S(W,DT(W))^{ac} & = S(W) h(W,\nabla T(W)) 
\\
& = S(W) \nabla T(W) \a(T(W),\nabla T(W)) = S(W) \nabla T(W) \a(W,\nabla T(W))
\end{split}
\]
(and in the same way for the other term) we have that
$$
h_S(W,DT(W))^{ac} + h_T(W,DS(W))^{ac} = \a(W,\nabla W) \nabla (S(W)T(W)).
$$
Thus, the inequality for the absolutely continuous parts reduces to
$$
\int_{Q_T}  \phi T(W) S(W) \div \z \ dxdt - \int_{Q_T}  \phi W_t T(W) S(W) \ dxdt \le 0.
$$
Then it suffices to show that 
$$
W_t \le \div \a(W,\nabla W)\quad \mbox{a.e. in} \ B(t), \quad \mbox{for a.e.}\ t \in (0,T).
$$
Now we discuss the inequality relating the singular parts (note that when $\gamma(t)=0$ there is no singular part at all, due to the fact that $\varphi(0)=0$). For that we compute
$$
[\z T(W)S(W) \cdot \nu^B] = \varphi(\gamma(t)) T(\gamma(t))S(\gamma(t))
$$
using condition \emph{\ref{tangencia}}. Then the inequality for the singular parts is equivalent to
$$
 - J_{TS\varphi'}(h(t)) \ge - J_{TS}(h(t)) C'(t) \quad \mbox{for a.e.}\ t \in (0,T).
$$
Thanks to condition \emph{\ref{tezeta}}, this is automatically fulfilled.
\end{proof}

In this way we are able to sharpen and extend the program that was introduced in \cite{ARMAsupp}.  We construct now sub-solutions spreading at any prefixed rate strictly lower than $s$. This is crucial as it is less demanding on $\psi$ to construct such than to construct sub-solutions attaining the rate given by $s$, see Remark \ref{c=s} below.

\begin{Proposition}
\label{1dcont}
Let $d=1$ and let $\psi$ satisfy Assumptions \ref{as1}. Let $c<s$ and $R>0$. Then the following statements hold:

\emph{1)} there exists some $A>0$ (depending on $\psi$, $c/s,\,L,\, R$) 
such that
$$
W(t,x)= e^{-At} \sqrt{R^2(t)-|x|^2} \chi_{B(0,R(t))},\quad R(t)=R+ct
$$
is a sub-solution of \eqref{template} in $Q_T$ for every $T>0$.

\emph{2)} Let $0<\theta<1$ and assume that 
\begin{equation}
\label{tetatrans}
\lim_{r \to \infty} r^\frac{1}{1-\theta} \psi'(r) = 0.
\end{equation}
Fix $\gamma_0>0$. Then there exists some $A>0$ (depending on $\psi$, $c/s,\, L,\, \gamma_0,\, R$) 
 such that
$$
W(t,x)= \left\{e^{-At} (R^2(t)-|x|^2)^\theta+\gamma_0 e^{-At}\right\} \chi_{B(0,R(t))},\quad R(t)=R+ct
$$
is a sub-solution of \eqref{template} in $Q_T$ for every $T>0$.
\end{Proposition}
\begin{proof}
Let us define 
$$
W(t,x)= \left(\alpha(t) (R^2(t)-|x|^2)^\theta+\gamma(t)\right) \chi_{B(0,R(t))},\quad R(t)=R+ct
$$
for some functions $\alpha,\gamma >0$ to be determined, such that $\alpha', \gamma' \le 0$.
Thanks to Proposition \ref{equi} we can restrict ourselves to check that $W_t \le (s W\psi(W_x/W))_x$ at $B(0,R(t))$ for each $t>0$. We are to show that
\begin{equation}
\label{s1}
W_t \le s W_x \psi(W_x/W) + s  \psi'( W_x/W) \left\{W_{xx}- (W_x)^2/W \right\}
\end{equation}
holds at $B(0,R(t))$ for every $t>0$. Neglecting the factor $\chi_{B(0,R(t))}$ in what follows, we compute
\[
\begin{split}
W_t & = \alpha'(t) (R^2(t)-|x|^2)^\theta + \frac{2 \theta \alpha(t) c R(t)}{(R^2(t)-|x|^2)^{1-\theta}} + \gamma'(t),
\\
W_x  & =- \frac{2 \theta\alpha(t) x}{(R^2(t)-|x|^2)^{1-\theta}},
\\
W_{xx} & = \frac{2 \theta (2 \theta-1) \alpha(t) x-2 \theta \alpha(t)R^2(t)}{(R^2(t)-|x|^2)^{2-\theta}}.
\end{split}
\]
 Due to parity, it suffices to show \eqref{s1} only for non-negative values of $x$.

 \emph{Case 1)} To prove the first statement we set $\gamma(t)=0$ and $\theta = 1/2$. Then we introduce $\lambda \in [0,1)$ so that $x = \lambda R(t)$. In terms of $\lambda$, \eqref{s1} reads now as follows:
\[
\begin{split}
\alpha'(t) & R(t) \sqrt{1-\lambda^2} + \frac{c \alpha(t)}{\sqrt{1-\lambda^2}}
\le -\frac{s \lambda \alpha(t)}{\sqrt{1-\lambda^2}} \psi \left(- \frac{ \lambda}{R(t) (1-\lambda^2)} \right)
\\
& + s  \psi'\left(- \frac{ \lambda}{R(t) (1-\lambda^2)}\right) \left\{-\frac{\alpha(t)(1+\lambda^2)}{R(t) (1-\lambda^2)^{3/2}} \right\}.
\end{split}
\]
This can be rearranged as
\[
\begin{split}
\frac{\alpha'(t)}{\alpha(t)} & \le - \frac{s}{R^2(t)} \frac{1+\lambda^2}{(1-\lambda^2)^2} \psi' \left( \frac{ \lambda}{R(t)(1-\lambda^2)}\right) 
\\
& + \frac{1}{R(t)(1-\lambda^2)} \left( s \lambda \psi\left( \frac{ \lambda}{R(t)(1-\lambda^2)}\right) -c \right).
\end{split}
\]
We introduce a new variable $r:=r(\lambda)=\frac{\lambda}{R(t)(1-\lambda^2)}$. Note that when $\lambda$ varies from $0$ to $1$, $r$ varies from $0$ to $\infty$. Hence, it suffices to show that
$$
\frac{\alpha'(t)}{\alpha(t)} \le s \left( r \psi(r) - \frac{c r}{\lambda s} - \frac{1+\lambda^2}{\lambda^2}r^2 \psi'(r)\right)
$$
for any $\lambda \in [0,1)$. If we ensure that the right hand side can be bounded from below by some constant $-A$, then the choice $\alpha(t)=e^{-At}$ would suit our purposes. The combination of terms at the right hand side is clearly bounded from below except maybe when $r \gg 1$. In order to see what happens in that case, we pick $\epsilon \in (0,1)$ such that $1-\epsilon=c/s$. Let us write
$$
 r \psi(r) - \frac{c r}{\lambda s} = r \left(\psi(r) - \frac{1-\epsilon/2}{\lambda} \right) + \frac{\epsilon r}{2 \lambda}.
$$
Provided that $\lambda>1-\epsilon/2$, the first term above is non-negative for $r$ large enough. Thus, if we show that
$$
\lim_{\lambda \to 1} \frac{\epsilon r}{2 \lambda}- \frac{1+\lambda^2}{\lambda^2}r^2 \psi'(r) \ge -A
$$
for some $A>0$ we will be done. It suffices to ask for
$$
\lim_{\lambda \to 1} \frac{\epsilon \lambda}{1+\lambda^2} \ge \lim_{\lambda \to 1} r \psi'(r).
$$
Recall that $\psi'\ge 0$ and note that we need to ensure the above inequality independently of the actual value of $\epsilon$ (thus allowing to get $c<s$ as close to $s$ as desired). Then we must impose the following condition:
$$
 \lim_{r \to \infty} r \psi'(r)=0.
$$
But this is automatically satisfied {thanks to Assumption \ref{as1}.\emph{\ref{1dcaida}}}. In this way our first statement follows.

 \emph{Case 2)} We introduce $\lambda \in [0,1)$ so that $x = \lambda R(t)$. In terms of $\lambda$, \eqref{s1} reads:
\[
\begin{split}
\alpha'(t) R^{2 \theta}(t) &(1-\lambda^2)^\theta + \frac{2 \theta c \alpha(t) R(t)}{R^{2-2\theta}(t) (1-\lambda^2)^{1-\theta}}+\gamma'(t)
\\
\le & -\frac{2 \theta s \lambda \alpha(t)R(t)}{R^{2-2\theta}(t) (1-\lambda^2)^{1-\theta}} \psi (I_1)
\\
&+ s \psi'(I_1)
\times \left\{-\frac{2 \theta (2 \theta-1)\alpha(t)\lambda^2 R^2(t)- 2 \theta \alpha(t) R^2(t)}{R^{4-2\theta}(t) (1-\lambda^2)^{2-\theta}}\right.
\\
& \left.-\frac{4 \theta^2 \alpha^2(t) \lambda^2R^2(t)}{R^{4-4\theta}(t) (1-\lambda^2)^{2-2\theta}[\alpha(t)R^{2 \theta}(t) (1-\lambda^2)^{\theta}+\gamma(t)]} \right\},
\end{split}
\]
being 
$$
I_1=- \frac{2 \theta \alpha(t) \lambda R(t)}{\alpha(t) R^2(t) (1-\lambda^2)+\gamma(t)R^{2-2\theta}(t)(1-\lambda^2)^{1-\theta}}.
$$ 
Thus, we will be done if we are able to check the following inequality:
\begin{equation}
\begin{array}{l}
\label{Rt}
\displaystyle \frac{\alpha'}{\alpha}\le \frac{s}{R^{2\theta}(t) (1-\lambda^2)^\theta} \left\{-\frac{\gamma'}{\alpha s} + \frac{2 \theta}{R^{1-2\theta}(t) (1-\lambda^2)^{1-\theta}}\right.\left(\lambda \psi(I_2) - \frac{c}{s} \right) 
\\ \\
 \displaystyle   + \psi'(I_2) \!\! \left.\left(\frac{2 \theta[(2 \theta -1)\lambda^2-1]}{R^{2-2\theta}(t)(1-\lambda^2)^{2-\theta}}-\! \frac{4 \theta^2 \lambda^2}{R^{2-2\theta}(t)(1-\lambda^2)^{2-\theta}+ R^{2-4 \theta}(t)(1-\lambda^2)^{2-2\theta} \frac{\gamma}{\alpha}} \right)\!\right\}
\end{array}
\end{equation}
with 
$$
I_2=-I_1=\frac{2 \theta \lambda}{R(t)(1-\lambda^2)+\frac{\gamma}{\alpha} R^{1-2 \theta}(t)(1-\lambda^2)^{1-\theta}}.
$$ 
As in the previous case, it suffices to ensure that the right hand side of the previous inequality is bounded from below by some constant $-A$. Let us choose $\gamma(t)=\gamma_0\alpha(t)$ with $\gamma_0>0$. Now we let $c/s=1-\epsilon$ for $\epsilon \in (0,1)$ and decompose
$$
\lambda \psi(I_2)-c/s=\lambda \left(\psi(I_2) - \frac{1-\epsilon/2}{\lambda}\right) +\epsilon/2. 
$$
The first term above is non-negative for $\lambda$ close enough to $1$. Taking this into account, it suffices to have 
\[
\begin{split}
\lim_{\lambda \to 1}  \psi'(I_2) & \left(\frac{2 \theta[(2 \theta -1)\lambda^2-1]}{R^{2-2\theta}(t)(1-\lambda^2)^{2-\theta}}- \frac{4 \theta^2 \lambda^2}{R^{2-2\theta}(t)(1-\lambda^2)^{2-\theta}+ R^{2-4 \theta}(t)(1-\lambda^2)^{2-2\theta} } \right)
\\
& -\frac{\alpha' \gamma_0}{\alpha s}+ \lim_{\lambda \to 1} \frac{\epsilon \theta}{R^{1-2\theta}(t)(1-\lambda^2)^{1-\theta}} \ge 0
\end{split}
\]
in order to ensure that the right hand side of \eqref{Rt} is bounded from below by some constant $-A$. Neglecting the second term above causes no loss of generality. Hence, we ask for
\[
\begin{split}
\lim_{\lambda \to 1}  \frac{\epsilon \theta}{R^{1-2\theta}(t)(1-\lambda^2)^{1-\theta}}
\ge \lim_{\lambda \to 1} & \left\{\frac{2 \theta }{R^{2-4\theta}(t)(1-\lambda^2)^{2-2\theta}} \psi'(I_2) \right.
\\
& \times \left.\left(\frac{2 \theta \lambda^2}{R^{2\theta}(t) (1-\lambda^2)^\theta +1}- \frac{(2 \theta -1)\lambda^2-1}{R^{2\theta}(t)(1-\lambda^2)^\theta} \right)\right\}.
\end{split}
\]
We notice again that the right hand side above is non-negative and that we need to ensure the above inequality independently of the actual value of $\epsilon$. Then the following condition must be imposed:
$$
\lim_{\lambda \to 1}  \psi'\left(\frac{1}{(1-\lambda^2)^{1-\theta}}\right) \frac{1}{1-\lambda^2}=0.
$$
This is the same as \eqref{tetatrans}, thus our statement is granted.
\end{proof}
\begin{Remark}
\label{c=s}
{\rm
Examining carefully the proof of the previous statement we note the following:
\begin{enumerate}
\item Provided that 
$$
 d(r) =O(1/r)\quad \mbox{and}\quad \psi'(r) =O(1/r^2) \quad \mbox{as}\quad r \to \infty, 
 $$
we can take $c=s$ in the first point of Proposition \ref{1dcont}.

\item  Provided that 
\begin{equation}
\label{disccond}
 d(r) =O(1/r)\quad \mbox{and}\quad \psi'(r) =O(r^\frac{\theta-2}{1-\theta}) \quad \mbox{as}\quad r \to \infty, 
\end{equation}
we can take $c=s$ in the second point of Proposition \ref{1dcont}.
\end{enumerate}
There is a value of $\theta$ such that \eqref{disccond} holds for every model of the form \eqref{pmodelos}, except for the case of Wilson's model \eqref{Wilson}.
}
\end{Remark}
Some of these results can be extended to higher dimensions under Assumptions \ref{giso}.
\begin{Proposition}
\label{ddcont}
Let $d>1$ and let $\psi$ satisfy Assumptions \ref{as1} and \ref{giso}. Let $c<s$ and $R>0$. Assume in addition that
\begin{equation}
\label{gprimacond}
\lim_{r \to \infty} r^2 g'(r)=-1.
\end{equation}
Then, there exists some $A>0$ (depending on $g$, $c/s,\, L,\, R$) such that
$$
W(t,x)= e^{-At} \sqrt{R^2(t)-|x|^2} \chi_{B(0,R(t))},\quad R(t)=R+ct
$$
is a sub-solution of \eqref{template} in $Q_T$ for every $T>0$.
\end{Proposition}
\begin{proof}
Let us define 
$$
W(t,x)= \alpha(t) \sqrt{R^2(t)-|x|^2} \chi_{B(0,R(t))},\quad R(t)=R+ct
$$
for some function $\alpha$ to be determined, such that $\alpha'\le 0$. Thanks to Proposition \ref{equi} we can restrict ourselves to check that $W_t \le (s W\psi(W_x/W))_x$ at $B(0,R(t))$ for each $t>0$. We are to show that
\begin{equation}
\label{s2}
W_t \le s \Delta W g\left(\left|\frac{\nabla W}{W}\right|\right) + s  g'\left(\left|\frac{\nabla W}{W}\right|\right) \left(\frac{\nabla W D^2 W \nabla W^T}{W|\nabla W|} - \frac{|\nabla W|^3}{W^2} \right)
\end{equation}
holds at $B(0,R(t))$ for every $t>0$, being $D^2 W$ the Hessian matrix of $W(t)$. Neglecting the factor $\chi_{B(0,R(t))}$ in what follows, we compute
\[
\begin{split}
W_t & = \alpha'(t) \sqrt{R^2(t)-|x|^2} + \frac{\alpha(t) c R(t)}{\sqrt{R^2(t)-|x|^2}},
\\
\partial_i W &= - \frac{\alpha(t) x_i}{\sqrt{R^2(t)-|x|^2}},\quad \frac{\nabla W}{W}= - \frac{x}{R^2(t)-|x|^2},
\\
\partial_{ij}^2 W &= -\frac{\alpha(t)[R^2(t)\delta_{ij} - |x|^2 \delta_{ij} + x_ix_j]}{(R^2(t)-|x|^2)^{3/2}},\quad \Delta W = - \frac{\alpha(t)[d R^2(t) +(1-d) |x|^2]}{(R^2(t)-|x|^2)^{3/2}}.
\end{split}
\]
We substitute into \eqref{s2} to obtain, after rearranging a bit,
\[
\begin{split}
\frac{\alpha'(t)}{\alpha(t)}\le & -\frac{cR(t)}{R^2(t)-|x|^2} - s g\left(\frac{|x|}{R^2(t)-|x|^2}\right)\frac{dR^2(t)+(1-d)|x|^2}{(R^2(t)-|x|^2)^2}
\\
& -s g'\left(\frac{|x|}{R^2(t)-|x|^2}\right) \frac{|x|^3+|x|R^2(t)}{(R^2(t)-|x|^2)^3}.
\end{split}
\]
This depends on $x$ only through $|x|$. Then we introduce $\lambda \in [0,1)$ so that $|x| = \lambda R(t)$. In terms of $\lambda$, the inequality to be satisfied reads:
\begin{equation}
\nonumber
\begin{array}{rl}
\displaystyle \frac{\alpha'(t)}{\alpha(t)}\le   \frac{1}{R(t)(1-\lambda^2)} &  \displaystyle \left\{ -c -s \frac{d+(1-d)\lambda^2}{R(t)(1-\lambda^2)} g \left( \frac{\lambda}{R(t)(1-\lambda^2)}\right) 
\right.
\\ \\
& \displaystyle \left.
 - s  \frac{\lambda^3 + \lambda}{R^2(t)(1-\lambda^2)^2}g'\left( \frac{\lambda}{R(t)(1-\lambda^2)}\right)\right\}.
\end{array}
\end{equation}
Our aim is to find a lower bound for the right hand side in terms of some constant $-A$, which would imply our result. Let us write $I$ for the term inside braces; to get such a bound, it suffices to show that  $\lim_{\lambda \to 1}I> 0$. In fact, 
$$
\lim_{\lambda \to 1} I = -c -s - s \lim_{\lambda \to 1} \frac{\lambda + \lambda^3}{R^2(t)(1-\lambda^2)^2}g'\left( \frac{\lambda}{R(t)(1-\lambda^2)}\right)= s-c>0 
$$
due to (\ref{gprimacond}) and property \eqref{ginfinito} in Remark \ref{gprops}. This concludes the proof.
\end{proof}
\begin{Remark}
{\rm
Note that \eqref{gprimacond} is satisfied by every equation of the form \eqref{pmodelos}.
}
\end{Remark}
The previous results allow us to track the evolution of the support in the same vein as in \cite{ARMAsupp}.
\begin{Corollary}
\label{c62}
(\emph{Evolution of the support})
Let $\psi$ satisfy Assumptions \ref{as1}. Let $C\subset \R^d$ be an open set and let $0\le u_0 \in (L^1\cap L^\infty)(\R^d)$ with support equal to $\overline{C}$. Let $u(t)$ be the entropy solution of \eqref{template} with $u_0$ as initial datum. Assuming that
\begin{equation}
\label{local_sep}
\mbox{for any closed set}\ F\subset C,\ \mbox{there is}\ \alpha_F >0\, \mbox{such that}\ u_0(x) \ge \alpha_F \, \forall x \in F,
\end{equation}
and either $d=1$ or Assumptions \ref{giso} holds together with \eqref{gprimacond}, then
$$
\mbox{supp} \, u(t) = cl\left(\mbox{supp} \, u_0 \oplus B(0,st)\right).
$$
\end{Corollary}
\begin{proof}
This is a combination of Proposition \ref{1dcont} (resp. Proposition \ref{ddcont}) and Corollary \ref{c61}. Note that both are invariant under spatial translations. Condition \eqref{local_sep} ensures that for each $y \in C$ we can find a suitable radius and a suitable height in order to apply the first point of Proposition \ref{1dcont} (resp. Proposition \ref{ddcont}) with a sub-solution centered at $y$ (as argued in the proof of Theorem 4 in \cite{ARMAsupp}), whose velocity $c$ can be chosen as close to $s$ as desired.
\end{proof}
\begin{Corollary}
(\emph{Persistence of discontinuous interfaces})
Let $d=1$ and let $\psi$ satisfy Assumptions \ref{as1}. Let $0\le u_0 \in (L^1\cap L^\infty)(\R)$ be supported on a bounded interval $[a,b]$. Let $u(t)$ be the entropy solution of \eqref{template} with $u_0$ as initial datum. Assuming that there exist some $\epsilon,\alpha >0$ such that $u_0(x)>\alpha>0$ for every $x \in (b-\epsilon,b)$ and that \eqref{disccond} holds, the left lateral trace of $u(t)$ at $x=b+st$ is strictly positive for every $t>0$. A similar statement holds for the left end of the support.
\end{Corollary}
\begin{proof}
This is similar to the previous one, but we use the second statement in Proposition \ref{1dcont} this time, taking $c=s$ thanks to Remark \ref{c=s}. Under the present assumptions we may choose $R=\epsilon/2$ and we will be able to find some values $\gamma_0, A>0$ and $0<\theta<1$ such that the corresponding sub-solution in Proposition \ref{1dcont} centered at $x=b-\epsilon/2$ lies below of $u_0$ for $t=0$. Thanks to Theorem \ref{theocomp}, $u(t)\ge \gamma_0 e^{-At}$ for  a.e. $x \in (b-\epsilon,b)\oplus B(0,st)$ and any $t>0$. This implies in particular that $u(t) \in BV ((b-\epsilon,b)\oplus B(0,st))$. Thus, we can compute the left lateral trace at $x=b+st$ as
$$
u(t,b+st)^- = \lim_{\lambda \to 0} \frac{1}{\lambda}\int_{b+st-\lambda}^{b+st} u(t,r)\, dr \ge \gamma_0 e^{-At}
$$
for any $t>0$.
\end{proof}


\section{Rankine--Hugoniot relations}
\label{rhcond}
The idea of this section is to generalize in a suitable way some results in \cite{leysalto, Vicent2013}. This will provide a proof for \emph{\ref{p3}} in Theorem \ref{Main2}. In so doing we will notice that such results hold for a class of equations which is wider than \eqref{template2}. In fact, in order that the main results in this section hold, what is really essential is that a function $\varphi$ can be defined by means of \eqref{deffi} satisfying a number of suitable properties. No further structure assumptions need to be imposed on the flux. The main results below are Proposition \ref{7.1}, which reformulates the entropy inequalities \eqref{eineq} as separate requirements on the jump and Cantor parts of the spatial derivative (a fact that was observed in greater generality in \cite{Vicent2013}), and Proposition \ref{8.1}, stating that the ``jump part'' of the entropy inequalities \eqref{eineq} is fulfilled if the flux at both sides of the discontinuity satisfies a certain constraint (encoding essentially the fact that contact angles must be vertical) and, given that this holds, phrasing the Rankine--Hugoniot relation in terms of $\varphi$.

We start by introducing some notation suited to this purpose (see also \cite{leysalto}). Assume that $u \in BV_{loc}(Q_T)$. Let $\nu:=\nu_u=(\nu_t,\nu_x)$ be the unit normal to the jump set of $u$ and $\nu^{J_{u(t)}}$ the unit normal to the jump set of $u(t)$. We write $[u](t,x):=u^+(t,x)-u^-(t,x)$ for the jump of $u$ at $(t,x)\in J_u$ and $[u(t)](x):=u(t)^+(x)-u(t)^-(x)$ for the jump of $u(t)$ at the point $x \in J_{u(t)}$. We assume that $u^+>u^-$ in what follows (this determines if $\nu_x$ points inwards or outwards according to the conventions on Subsection \ref{sect:bv}); we also assume $u^-\ge 0$. The following result was proved in \cite{leysalto}.
\begin{Lemma}
Let $u\in BV_{loc}(Q_T)$ and let $\z \in L^\infty([0,T]\times \R^d,\R^d)$ be such that $u_t = \div \z$ in $\mathcal{D}'(Q_T)$. Then
$$
\mathcal{H}^d \left(\{(t,x) \in J_u/ \nu_x(t,x)=0\}\right)=0.
$$
\end{Lemma}
\begin{Definition}
Let $u \in BV_{loc}(Q_T)$ and let $\z \in L^\infty ([0,T]\times \R^d,\R^d)$ be such that $u_t = \div \z$ in $\mathcal{D}'(Q_T)$. We define the speed of the discontinuity set of $u$ as $\v(t,x) = \frac{\nu_t(t,x)}{|\nu_x(t,x)|}\, \mathcal{H}^d$-a.e. on $J_u$.
\end{Definition}
Next we quote a result encoding the Rankine--Hugoniot conditions that can  be found in \cite{leysalto} too.
\begin{Proposition}
\label{conrh}
Let $u \in BV_{loc}(Q_T)$ and let $\z \in L^\infty([0,T]\times \R^d,\R^d)$ be such that $u_t = \div \z$. For a.e. $t\in (0,T)$ we have
$$
[u(t)](x) \v(t,x) = [[\z \cdot \nu^{J_{u(t)}}]]_{+-}\quad \mathcal{H}^{d-1}-\mbox{a.e. in}\ J_{u(t)},
$$
where $[[\z \cdot \nu^{J_{u(t)}}]]_{+-}$ denotes the difference of traces from both sides of $J_{u(t)}$.
\end{Proposition}
The following statement is a particular case of Proposition 6.8 in \cite{Vicent2013}.
\begin{Proposition}
\label{7.1}
Let $u \in C([0,T];L^1(\R^d))\cap BV_{loc}(Q_T)$. Assume that $u_t=\div \z$ in $\mathcal{D}'(Q_T)$, where $\z = \a(u,\nabla u)$. Assume also that $u_t(t)$ is a Radon measure for a.e. $t>0$. Let $\varphi$ defined by \eqref{deffi} be a locally Lipschitz continuous function such that $\varphi(0)=0$. Then $u$ is an entropy solution of \eqref{template2} if and only if for any $(T,S) \in \mathcal{TSUB}$ (for any $(T,S) \in \mathcal{TSUB}\cup \mathcal{TSUPER}$) we have
$$
h_S(u,DT(u))^c+h_T(u,DS(u))^c \le (\z(t,x)\cdot D(T(u)S(u)))^c
$$
and for almost any $t>0$ the inequality
\begin{equation}
\begin{array}{rl}
\label{entcondsimple}
\displaystyle [ST \varphi(u(t))]_{+-} & \displaystyle- \ [J_{TS\varphi'}(u(t))]_{+-} 
\\ \\
& \displaystyle \le - \v [J_{TS}(u(t))]_{+-} + [[\z(t)\cdot \nu^{J_{u(t)}} ]T(u(t))S(u(t))]_{+-}
\end{array}
\end{equation}
holds $\mathcal{H}^{d-1}$-a.e. on $J_{u(t)}$.
\end{Proposition}
\begin{proof}
The same proof as in Proposition 7.1 of \cite{leysalto} can be used. The only noticeable difference is found when extracting jump parts from the entropy inequalities  (\ref{eineq}). Here the property $\varphi(0)=0$ is needed in order to ensure that
$$
([J_{S\varphi T'}(u(t))]_{+-}+[J_{T\varphi S'}(u(t))]_{+-}) \mathcal{H}^{d-1}|_{J_{u(t)}}dt
$$
agrees with
$$ 
([ST \varphi(u(t))]_{+-}-[J_{TS\varphi'}(u(t))]_{+-})\mathcal{H}^{d-1}|_{J_{u(t)}}dt.
$$
Then the rest of the proof goes as in \cite{leysalto}.
\end{proof}

Now we state and prove the main result of the Section, which generalizes Proposition 8.1 in \cite{leysalto} (see also \cite{Vicent2013} for a similar statement concerning a related class of flux-limited equations). 
\begin{Proposition}
\label{8.1}
Let $u \in C([0,T];L^1(\R^d))$ be the entropy solution of \eqref{template2} with $0\le u(0)=u_0 \in L^\infty(\R^d)\cap BV(\R^d)$. Assume that $u \in BV_{loc}(Q_T)$. Assume further that $\varphi$ defined by (\ref{deffi}) is a convex, non-negative function such that $\varphi(0)=0$. Then the entropy conditions \eqref{entcondsimple} hold if and only if for almost any $t \in (0,T)$
\begin{equation}
\label{entcondsimple2}
[\z \cdot \nu^{J_{u(t)}}]_+ = \varphi(u^+(t)) \quad \mbox{and} \quad [\z \cdot \nu^{J_{u(t)}}]_- = \varphi(u^-(t))
\end{equation}
hold $\mathcal{H}^{d-1}$-a.e. on $J_{u(t)}$. Moreover the speed of any discontinuity front is 
\begin{equation}
\label{vformula}
\v=\frac{\varphi(u^+(t))-\varphi(u^-(t))}{u^+(t)- u^-(t)}.
\end{equation}
\end{Proposition}
\begin{proof}
The proof is a suitable generalization of that given for Proposition 8.1 in \cite{leysalto}. Recall that the Rankine--Hugoniot conditions stated in Proposition \ref{conrh} are 
$$
 \v [u]_{+-} = [[\z \cdot \nu^{J_{u(t)}}]_{+-}.
$$
Let us show that \eqref{entcondsimple} implies \eqref{entcondsimple2}. For that we let $\epsilon >0$ be such that $u^- < u^+-\epsilon <u^+$ and we choose $(S,T)\in \mathcal{TSUB}$ so that $S(r)T(r)=(r-(u^+-\epsilon))^+$. Then we compute:
\begin{enumerate}
\item $[ST \varphi(u(t))]_{+-}= \epsilon \varphi(u^+)$,
\item $[J_{TS}(u(t))]_{+-}= \frac{\epsilon^2}{2}$,
\item $[[\z(t) \cdot \nu^{J_{u(t)}}]T(u(t))S(u(t))]_{+-}=\epsilon [\z(t) \cdot \nu^{J_{u(t)}}]_{+}$,
\item $[J_{TS\varphi'}(u(t))]_{+-}= \int_{u^+-\epsilon}^{u^+}(r-(u^+-\epsilon))\varphi'(r) \, dr \le \epsilon (\varphi(u^+)-\varphi(u^-))\le C\epsilon^2$ for some $C>0$, as $\varphi$ is locally Lipschitz.
\end{enumerate}
Then \eqref{entcondsimple} is written as
\begin{equation}
\nonumber
\epsilon (\varphi(u^+) -  [\z(t) \cdot \nu^{J_{u(t)}}]_{+}) \le C\epsilon^2-\frac{\epsilon^2}{2} v,
\end{equation}
which is a contradiction unless $ [\z(t) \cdot \nu^{J_{u(t)}}]_{+}= \varphi(u^+)$ (as $| [\z(t) \cdot \nu^{J_{u(t)}}]_{+}|\le \varphi(u^+)$ clearly holds). We show that $ [\z(t) \cdot \nu^{J_{u(t)}}]_{-}= \varphi(u^-)$ in a similar way. Using the Rankine--Hugoniot condition, the speed of the front is given by
$$
\v = \frac{[\z \cdot \nu^{J_{u(t)}}]_{+}- [\z \cdot \nu^{J_{u(t)}}]_{-}}{u^+- u^-} = \frac{\varphi(u^+)-\varphi(u^-)}{u^+- u^-}.
$$

Let us show now the converse implication. Thanks to \eqref{entcondsimple2} we may write
\[
\begin{split}
[[\z \cdot \nu^{J_{u(t)}}]T(u)S(u)]_{+-} & = [\z \cdot \nu^{J_{u(t)}}]_+T(u^+)S(u^+) - [\z \cdot \nu^{J_{u(t)}}]_-T(u^-)S(u^-)
\\
& = \varphi(u^+)T(u^+)S(u^+) - \varphi(u^-)T(u^-)S(u^-)= [ST\varphi(u(t))]_{+-}.
\end{split}
\]
Thus, we recast \eqref{entcondsimple} as
\begin{equation}
\label{entcondsimple3}
\frac{\varphi(u^+)-\varphi(u^-)}{u^+- u^-} [J_{TS(u(t))}]_{+-} \le [J_{TS\varphi'}]_{+-}.
\end{equation}
Let us show that \eqref{entcondsimple3} holds for any $(T,S)\in \mathcal{TSUB}\cup \mathcal{TSUPER}$. As argued in \cite{leysalto}, to treat the case $(T,S)\in \mathcal{TSUB}$ it suffices to deal with $TS(r)=p(r)=\chi_{(d,\infty)}(r)$. There are several sub-cases to consider:
\begin{itemize}
\item $u^-\ge 0$ and $d \le u^+$. Then $[J_p(u(t))]_{+-}=[u]_{+-}$ and $[J_{p\varphi'(u(t))}]_{+-}=[\varphi(u)]_{+-}$. Thus \eqref{entcondsimple3} holds.

\item $u^-\ge 0$ and $u^-<d \le u^+$. We compute $[J_p(u(t))]_{+-}=u^+-d$ and $[J_{p\varphi'(u(t))}]_{+-}=\varphi(u^+)-\varphi(d)$. Then \eqref{entcondsimple3} is equivalent to 
$$
\frac{\varphi(u^+)-\varphi(u^-)}{u^+- u^-} (u^+-d)\le \varphi(u^+) -\varphi(d)
$$
which in turn holds because $\varphi$ is convex.
\item $u^-\ge 0$ and $d > u^+$. Then $[J_p(u(t))]_{+-}=[J_{p\varphi'(u(t))}]_{+-}=0$. Hence \eqref{entcondsimple3} is trivially satisfied.
\end{itemize}
Similarly, to treat the case $(T,S)\in \mathcal{TSUPER}$ it suffices to deal with $TS(r)=p(r)=c+c'\chi_{(d,\infty)}(r),\, c\le 0,\, 0\le c'\le |c|$. Again, we consider the various sub-cases:
\begin{itemize}
\item $u^-\ge 0$ and $d \le u^+$. Then $[J_p(u(t))]_{+-}=(c+c')[u]_{+-}$ and $[J_{p\varphi'(u(t))}]_{+-}=(c+c')[\varphi(u)]_{+-}$. Thus \eqref{entcondsimple3} holds.

\item $u^-\ge 0$ and $u^-<d \le u^+$. We compute $[J_p(u(t))]_{+-}=c[u]_{+-}+c'(u^+-d)$ and $[J_{p\varphi'(u(t))}]_{+-}=c[\varphi(u)]_{+-}+c'(\varphi(u^+)-\varphi(d))$. Then \eqref{entcondsimple3} is equivalent to 
$$
\frac{\varphi(u^+)-\varphi(u^-)}{u^+- u^-}\le \frac{\varphi(u^+) -\varphi(d)}{u^+-d}
$$
which in turn holds because $\varphi$ is convex.

\item $u^-\ge 0$ and $d > u^+$. This time $[J_p(u(t))]_{+-}=c [u]_{+-}$ and $[J_{p\varphi'(u(t))}]_{+-}=c [\varphi(u)]_{+-}$. Hence  \eqref{entcondsimple3} is satisfied.
\end{itemize}
\end{proof}
\begin{Remark}
\label{contacto}
{\rm 
Under some additional assumptions we may derive from \eqref{vformula} a vertical contact angle condition, as pointed out in \cite{leysalto}. For that we assume that for $\mathcal{H}^d$-almost $x \in J_{u}$ there is a ball $B_x$ centered at $x$ such that either (a) $u|B_x \ge \alpha >0$ or (b) $J_u \cap  B_x$ is the graph of a Lipschitz function with $B_x \backslash J_u = B_x^1 \cup B_x^2$, where $B_x^1, B_x^2$ are open and connected and $u \ge \alpha >0 $ in $B_x^1$, while the trace of $u$ on $J_u \cap \partial B_x^2$ computed from $B_x^2$ is zero. 
In both cases $[\psi(L\nabla u/u) \cdot \nu^{J_{u(t)}}]_+=1$ on $J_u \cap B_x$. If (a) holds, we also have $[\psi(L\nabla u/u) \cdot \nu^{J_{u(t)}}]_-=1$ on $J_u \cap B_x$. Provided that the Jacobian matrix of $\psi$ is not compactly supported, these relations imply in particular that $|\nabla u|=\infty$.
}
\end{Remark}
\begin{Remark}
{\rm 
In case that $u^-=0$, \eqref{vformula} reduces to $\v = \varphi(u^+)/u^+$. As $\varphi$ is convex, this is compatible with Corollary \ref{c61}.
}
\end{Remark}
Now we give a sufficient condition to ensure that $u_t$ is a Radon measure, which is required in order to use Proposition \ref{8.1}.
\begin{Proposition}
Let $0 \le u_0 \in (L^1 \cap L^\infty)(\R^d)$ and let $u(t)$ be the entropy solution of \eqref{template2} with $u_0$ as initial datum. If $\varphi$ is homogeneous of degree $m>1$, then for any $t>0$, $u_t(t)$ is a Radon measure in $\R^d$. Moreover $\|u_t(t)Ê\|_{\mathcal{M}(\R^d)}\le \frac{2}{(m-1)t}\|Êu_0\|_1$.
\end{Proposition}
\begin{proof}
This is a direct consequence of the results in \cite{BenilanNote}.
\end{proof}

Next we wonder about the admissible discontinuity gaps for a given speed.
\begin{Lemma}
\label{luegotellamo}
Assume that $\varphi$ is strictly convex. Then, given values $\v, u^+$ such that $\varphi(u^+)/u^+\le \v < \varphi'(u^+)$, there exists an unique value $u^- \in [0,u^+)$ such that relation \eqref{vformula} holds. 
\end{Lemma}
\begin{proof}
Given the value $u^+$, we look for values of $u^-$ such that  \eqref{vformula} holds (provided they exist). We consider the function
$$
\phi(x)= \frac{\varphi(u^+) - \varphi(x)}{u^+ - x},
$$
 defined for $x\in [0,u^+)$. It is easily seen that
 $$
\phi(0) = \varphi(u^+)/u^+ \quad \mbox{and} \quad \phi(u^+)= \varphi'(u^+).
$$
Next we compute
$$
\phi'(x)= \frac{\varphi(u^+) - (u^+-x)\varphi'(x) -\varphi(x)}{(u^+-x)^2} >0 \quad \mbox{for}\ x\in (0,u^+).
$$
Thus, $\phi$ is a bijection from $[0,u^+]$ to $[\varphi(u^+)/u^+, \varphi'(u^+)]$, which implies our result.
\end{proof}

Contrary to the situation depicted in the previous result, we have:
\begin{Corollary}
\label{zerocase}
Let $\varphi(z)=s z$. Then the only speed of propagation for discontinuity fronts that is allowed is precisely $s$, while any values of $u^+>u^-\ge 0$ are admissible for such a discontinuity. 
\end{Corollary}



\section*{Acknowledgments}
The author acknowledges J.M.~Maz\'on for his useful comments on a previous version of this document and the anonymous referees of SIAM J. Math. Anal., which helped to improve greatly the contents of the paper with their comments. He also thanks F. Andreu, V. Caselles and J.M. Maz\'on for their support during these years. J. Calvo acknowledges partial support by a Juan de la Cierva grant of the spanish MEC, by the ``Collaborative Mathematical Research'' programme by ``Obra Social La Caixa'', by MINECO (Spain), project MTM2014-53406-R, FEDER resources, and Junta de Andaluc\'ia Project P12-FQM-954.



\begin{thebibliography}{99}
\bibitem{Agueh}
{\sc M.~Agueh}, {\em Existence of solutions to degenerate parabolic equations via the Monge-Kantorovich theory}, PhD Thesis, Georgia Tech, Atlanta, 2001.
%
%
\bibitem{Anzellotti1}
{\sc G.~Anzellotti}, {\em Pairings between measures and bounded functions
and compensated compactness},
Ann. di Matematica Pura ed Appl. IV, 135
(1983), pp.~93--318.

%
\bibitem{Ambrosio}
{\sc L.~Ambrosio, N.~Fusco and D.~Pallara}, {\em Functions of
Bounded Variation and Free Discontinuity Problems},
 Oxford Mathematical Monographs, 2000.
 
\bibitem{AmbrosioII}
{\sc L.~Ambrosio, N.~Gigli and  G.~Savar\'e}, {\em Gradient flows in metric spaces and in the spaces of probability measures},
 Lectures in Mathematics ETH Zurich, Birkhauser Verlag, Basel, 2005. 

\bibitem{Pisa}
{\sc F. Andreu, V. Caselles and J. Maz\'on},{\em A Strongly Degenerate Quasilinear Equation: the Elliptic Case} Ann. Scuola Norm. Sup. Pisa Cl. Sci. 5, Vol III (2004), pp.~555--587.

\bibitem{ACMEllipticFLDE}
{\sc F.~Andreu, V.~Caselles and J.M.~Maz\'on},
{\em A Strongly Degenerate Quasilinear  Elliptic Equation},
Nonlinear Analysis, TMA, 61 (2005), pp.~637--669.


\bibitem{ACMMRelat}
{\sc F.~Andreu, V.~Caselles and J.M.~Maz\'on}, {\em The Cauchy
Problem for a Strongly Degenerate Quasilinear Equation}, J. Europ.
Math. Soc., 7 (2005), pp.~361--393.

\bibitem{ACMregularity}
{\sc F.~Andreu, V.~Caselles and  J.M. Maz\'on}, {\em Some regularity results on the  `relativistic' heat
equation}, Journal of Differential Equations 245 (2008), pp.~3639--3663.

\bibitem{ARMAsupp}
{\sc F.~Andreu, V.~Caselles, J.M.~Maz\'on and S.~Moll,} {\em Finite propagation speed for limited flux diffusion equations}, Arch. Rat. Mech. Anal., 182 (2006), pp.~269--297.

\bibitem{TrMedia}
{\sc F.~Andreu, V.~Caselles, J.M.~Maz\'on and S.~Moll,} {\em A diffusion equation in transparent media}, J. Evol. Equ., 7 (2007), pp.~113--143. 
%
\bibitem{ACMSV} 
{\sc F.~Andreu, V.~Caselles, J.M.~Maz\'on, J.~Soler and M.~Verbeni}, {\em Radially symmetric solutions of a tempered diffusion equation. A porous media flux-limited case}, SIAM J. Math. Anal., 44 (2012), pp.~1019--1049.

\bibitem{BenilanNote}
{\sc Ph.~Benilan, M.G.~Crandall}, {\em Regularizing effects of homogeneous evolution equations}, in Contribution to Analysis and Geometry, 1981, pp.~23--39.

\bibitem{6son6}
{\sc Ph.~Benilan, L.~Boccardo, T.~Gallouet, R.~Gariepy, M.~Pierre and J.L.~Vazquez}, {\em An L1-Theory of Existence and Uniqueness of Solutions of Nonlinear Elliptic Equations}, Ann. Scuola Normale Superiore di Pisa, IV, Vol. XXII (1995), pp.~241--273.

\bibitem{BdPasso}
{\sc M.~Bertsch and R.~Dal Passo}, {\em Hyperbolic Phenomena in a Strongly Degenerate Parabolic Equation},
Arch. Rational Mech. Anal., 117 (1992), pp.~349--387.


\bibitem{Brenier1}
{\sc Y.~Brenier}, {\em Extended Monge-Kantorovich Theory},
In Optimal Transportation and Applications, Lectures given at the C.I.M.E. Summer School help
in Martina Franca, 
Lecture
Notes in Math. 1813, L.A. Caffarelli and S. Salsa, eds., Springer--Verlag, 2003, pp.~91--122.

%
\bibitem{2014}
{\sc J.~Calvo, J.~Campos, V.~ Caselles,O.~S\'anchez and J.~Soler},{\em Flux saturated porous media diffusion equations and applications}, to appear in J. Eur. Math. Soc. (JEMS)

\bibitem{Carrillo1} 
{\sc J.~Carrillo and P.~Wittbold}, {\em Uniqueness of Renormalized Solutions of Degenerate Elliptic-Parabolic
problems}, J. Differ. Equ., 156 (1999), pp.~93--121.



 \bibitem{leysalto}
{\sc V.~Caselles}, {\em On the entropy conditions for some flux limited diffusion equations}, J. Diff. Eqs., 250 (2011), pp.~3311--3348.

\bibitem{Vicent2013}
{\sc V.~Caselles}, {\em Flux limited generalized porous media diffusion equations}, Publ. Mat., 57 (2013), pp.~155--217. 
 

 \bibitem{Rosenau2003}
{\sc A.~Chertock, A.~Kurganov and P.~Rosenau},
 {\em Formation of discontinuities in flux-saturated degenerate parabolic equations},
 Nonlinearity, 16 (2003), pp.~1875--1898.
 
  \bibitem{Rosenau2005}
{\sc A.~Chertock, A.~Kurganov and P.~Rosenau},
 {\em On degenerate saturated-diffusion equations with convection},
 Nonlinearity, 18 (2005), pp.~609--630.
 
 \bibitem{KSlimitado}
{\sc A.~Chertock, A.~Kurganov, X.~Wang and Y.~Wu},
{\em On a chemotaxis model with saturated chemotactic flux}, 
Kinetic and Related Models, 5 (2012), pp.~51--95. 


%
\bibitem{Coulombel}
{\sc J.-F.~Coulombel, F.~Golse and T.~Goudon}, {\em Diffusion Approximation and Entropy-based Moment Closure for Kinetic Equations}, Asymptotic Analysis, 45 (2005), pp.~1--39.  	

\bibitem{CrandallLiggett}
{\sc M.G.~Crandall and T.M.~Liggett}, {\em Generation of Semigroups of Nonlinear
Transformations on General Banach Spaces},
 Amer. J. Math., 93 (1971), pp.~265--298.

\bibitem{Dalmaso}
{\sc G.~Dal Maso}, {\em Integral representation on $BV(\Omega)$ of $\Gamma$-limits of
variational integrals},
Manuscripta Math., 30 (1980), pp.~387--416.

\bibitem{DCFV}
{\sc V.~De Cicco, N.~Fusco and A.~Verde}, {\em On $L^1$-lower semicontinuity in $BV$},
J. Convex Analysis, 12 (2005), pp.~173--185.
%


\bibitem{Giacomelli}
{\sc L.~Giacomelli}, {\em Finite speed of propagation and waiting-time phenomena for degenerate parabolic equations with linear growth Lagrangian}. Preprint.

%

\bibitem{Jordan} 
{\sc R.~Jordan, D.~Kinderlehrer and F.~Otto}, {\em The variational formulation of the Fokker--Planck equation}, SIAM J. Math. Anal., 29 (1998), pp.~1--17.


\bibitem{Kruzhkov}
{\sc S.N.~Kruzhkov}, {\em First order quasilinear equations in several independent
variables},
 Math. USSR-Sb., 10 (1970), pp.~217--243.


\bibitem{Rosenau2006}
{\sc A.~Kurganov and P.~Rosenau}, {\em On reaction processes with saturating diffusion},
Nonlinearity, 19  (2006), pp.~171--193.


\bibitem{Levermore79}
{\sc C.D.~Levermore}, {\em A Chapman--Enskog approach to flux limited diffusion theory}, Technical Report, Lawrence Livermore Laboratory, UCID-18229 (1979). 	

\bibitem{Levermore81}
{\sc C.D.~Levermore and G.C.~Pomraning}, {\em A flux-limited diffusion theory}, The astrophysical journal,  248 (1981), pp.~321--334.  	
%
\bibitem{Levermore84}
{\sc C.D.~Levermore}, {\em Relating Eddington factors to flux limiters}, J. Quant. Spectrosc. Radiat. Transfer, 31  (1984), pp.~149--160.  	


\bibitem{McCann}
{\sc R.~Mc Cann and M.~Puel}, {\em Constructing a relativistic heat flow by transport
time step},
 Annales de l'Institut Henri Poincar\'e (C) Non Linear Analysis, 26  (2009), pp.~2539--2580.

\bibitem{Mihalas}
{\sc D.~Mihalas, B.~Mihalas}, {\em Foundations
of radiation hydrodynamics}, Oxford University Press, 1984.

\bibitem{Olson}
{\sc G.L.Olson, L.H.Auer and M.L.Hall}, {\em Diffusion, $P_{1}$, and other approximate forms of radiation transport},
 Journal of Quantitative Spectroscopy \& Radiative Transfer, 64 (2000), pp.~619--634.

 \bibitem{Rosenau92}
{\sc P.~Rosenau}, {\em Tempered Diffusion: A Transport Process with Propagating Front and Inertial Delay},
 Phys. Review A, 46 (1992), pp.~7371--7374.


\bibitem{Verbeni} 
{\sc M.~Verbeni, O.~S\'anchez, E.~Mollica, I.~Siegl-Cachedenier, A.~Carleton, I.~Guerrero, A.~Ruiz i Altaba and J.Soler}, {\em Morphogenetic action  through flux-limited spreading}, Phys. Life Rev., 10 (2013), pp.~457--475.



%

\end{thebibliography}
\end{document}